\documentclass[reqno, 11pt]{article}

\usepackage{enumerate}
\usepackage{needspace}
\usepackage{latexsym}
\usepackage{amsmath}
\usepackage{amsfonts}
\usepackage{amsthm}
\usepackage{amssymb,mathtools}
\usepackage{newlfont}
\usepackage{graphics}
\usepackage{color}
\usepackage{float}
\usepackage{diagbox}
\usepackage{tocloft}
\usepackage{titlesec}
\usepackage{booktabs,array,longtable,enumitem,microtype}
\usepackage{extpfeil}
\usepackage{centernot}

\usepackage[linesnumbered,ruled,vlined]{algorithm2e}
\usepackage{url}

\usepackage[margin=1in]{geometry}

\usepackage[T1]{fontenc}
\usepackage{lmodern}
\usepackage{microtype}
\usepackage{rotating}
\usepackage{multirow}
\usepackage{extarrows}
\usepackage[sort,compress,numbers]{natbib}
\usepackage[utf8]{inputenc}
\usepackage{xcolor}
\usepackage{listings}
\usepackage{aliascnt}
\usepackage[pagebackref,colorlinks=true,linkcolor=blue,citecolor=red,urlcolor=blue]{hyperref}
\usepackage[nameinlink,noabbrev,capitalize]{cleveref}
\numberwithin{equation}{section}

\setlist[enumerate]{label=\textup{(\roman*)},leftmargin=*}
\newtheorem{theorem}{Theorem}[section]
\newaliascnt{proposition}{theorem}
\newtheorem{proposition}[proposition]{Proposition}
\aliascntresetthe{proposition}
\newaliascnt{lemma}{theorem}
\newtheorem{lemma}[lemma]{Lemma}
\aliascntresetthe{lemma}
\newaliascnt{corollary}{theorem}
\newtheorem{corollary}[corollary]{Corollary}
\aliascntresetthe{corollary}

\theoremstyle{definition}
\newaliascnt{definition}{theorem}

\aliascntresetthe{definition}
\newaliascnt{remark}{theorem}
\newtheorem{remark}[remark]{Remark}
\aliascntresetthe{remark}
\newaliascnt{example}{theorem}
\newtheorem{example}[example]{Example}
\aliascntresetthe{example}
\newaliascnt{conjecture}{theorem}

\aliascntresetthe{conjecture}
\newaliascnt{question}{theorem}
\newtheorem{question}[question]{Question}
\aliascntresetthe{question}
\newaliascnt{problem}{theorem}

\aliascntresetthe{problem}
\newaliascnt{alg}{theorem}

\aliascntresetthe{alg}

\allowdisplaybreaks[4]

\SetKwInput{KwInput}{Input}                
\SetKwInput{KwOutput}{Output}              

\DeclareMathOperator{\conv}{conv}
\DeclareMathOperator{\cone}{cone}
\DeclareMathOperator{\aff}{aff}
\DeclareMathOperator{\htop}{ht}
\DeclareMathOperator{\nvol}{nVol}
\DeclareMathOperator{\vertices}{Vert}
\DeclareMathOperator{\Hilb}{Hilb}
\newcommand{\Z}{\mathbb Z}
\newcommand{\R}{\mathbb R}

\newcommand{\avec}{\boldsymbol a}
\newcommand{\bvec}{\boldsymbol b}
\newcommand{\eps}{\varepsilon}

\newcommand{\ceil}[1]{\left\lceil #1\right\rceil}
\newcommand{\floor}[1]{\left\lfloor #1\right\rfloor}

\title{Two families of Hermite normal form simplices}

\author{Feihu Liu$^{\color{blue} \dag}$, Jinlong Tang$^{\color{blue} \ddag}$, Sihao Tao$^{\color{blue} \P}$, and Zihao Zhang$^{\color{blue} \S}$
\\[2mm]
{\small $^{\color{blue} \dag}$ Center for Combinatorics, LPMC,}\\[-0.8ex]
{\small Nankai University, Tianjin 300071, P.R.~China}\\
{\small $^{\color{blue} \ddag, \P}$ School of Mathematical Sciences,}\\[-0.8ex]
{\small Capital Normal University, Beijing, 100048, P.R.~China}\\
{\small $^{\color{blue} \S}$ School of Mathematics and Statistics,}\\[-0.8ex]
{\small Beijing Institute of Technology, Beijing 102400, P.R.~China}\\
{\small {\color{blue} $^\dag$} Email address: \url{liufeihu7476@163.com}}\\
{\small {\color{blue} $^\ddag$} Email address: \url{jinlong\_tang@cnu.edu.cn}}\\
{\small {\color{blue} $^\P$} Email address: \url{sihao\_tao@cnu.edu.cn}}\\
{\small {\color{blue} $^\S$} Email address: \url{zihao-zhang@foxmail.com}}\\
}

\date{\today}

\begin{document}

\maketitle

\begin{abstract}
We study Ehrhart coefficients, the integer decomposition property, and unimodular triangulations for two families of Hermite normal form simplices considered by Bruckamp, Caicedo, and Juhnke. The first family consists of simplices of the form $S_{\boldsymbol{a}}=\mathrm{conv}(0,e_1,\ldots,e_{d-1},\boldsymbol{a})$, where $\boldsymbol{a}=(a_1,\ldots,a_{d-1},N)$, and the second consists of the simplices $T_{d,N}=\mathrm{conv}\bigl(0,e_1,\ldots,e_{d-2},(d-2,\ldots,d-2,d-1,0),(1,\ldots,1,N)\bigr)$. For the family $T_{d,N}$, we prove unimodality of the Ehrhart coefficients in every dimension and determine the log-concave and real-rooted cases. For the subfamily $S_{\boldsymbol{a}}$ with $\boldsymbol{a}=(N-q,\ldots,N-q,N)$, we characterize both the integer decomposition property and the existence of a regular unimodular triangulation via a congruence condition on a negative continued fraction. We also give an arithmetic test for the integer decomposition property and finite triangulation tests for $S_{\boldsymbol{a}}$ with arbitrary $\boldsymbol{a}$. These results address three open questions of Bruckamp, Caicedo, and Juhnke.
\end{abstract}

\noindent
\begin{small}
\end{small}

\noindent
\begin{small}
\end{small}

\tableofcontents

\section{Introduction}\label{sec:introduction}

Let $P\subset\R^d$ be a $d$-dimensional \emph{lattice polytope}, that is, the convex hull of finitely many points of $\Z^d$. For a nonnegative integer $t$, Ehrhart's theorem states that
\[L_P(t)=|tP\cap\Z^d|
\]
is a polynomial $L_P(t)$ of degree $d$; see \cite{Ehrhart} and \cite[Chapter~3]{BR}. 
This polynomial is called the \emph{Ehrhart polynomial} of $P$.
The coefficients of this polynomial reflect the geometry of $P$, but their signs and relative sizes are generally difficult to control. 

We use $\conv$ for convex hull and $e_1,\ldots,e_d$ for the standard
basis of $\R^d$. For integers $d,N\ge2$ and $0\le a_i<N$, set
\begin{equation}\label{eq:one-row-family}
 S_{\avec}=\conv(0,e_1,\ldots,e_{d-1},\avec),\qquad \avec=(a_1,\ldots,a_{d-1},N).
\end{equation}
These are the \emph{one-row Hermite normal form simplices}: in the matrix whose rows are their nonzero vertices, only the last row can differ from a row of the identity matrix. Our row convention for Hermite normal form is specified in \cref{sec:lattice-conventions}; it agrees with \cite[Section~2.1]{BBCHV}.
Every lattice simplex is unimodularly equivalent to a simplex in Hermite normal form \cite{Schrijver}.
The one-row family $S_{\avec}$ has been studied in \cite{BBCHV,BDHLS,HHLi}.

We shall also consider
\begin{equation}\label{eq:two-row-family}
 T_{d,N}=\conv\bigl(0,e_1,\ldots,e_{d-2},(d-2,\ldots,d-2,d-1,0),(1,\ldots,1,N)\bigr).
\end{equation}
For $d=2$, this means $T_{2,N}=\conv(0,(1,0),(1,N))$.
Bruckamp, Caicedo, and Juhnke proved that $T_{d,N}$ is Ehrhart positive, meaning that every coefficient of $L_{T_{d,N}}(t)$ is positive \cite[Theorem~5.12]{BCJ}. 
Subsequently, they posed the following open problem.

\begin{question}(\cite[Question~5.14]{BCJ})
Let $d\geq 2$. Is the Ehrhart polynomial of $T_{d,N}$ unimodal for all integers $N\ge 2$?
\end{question}

Recall that a real sequence $(c_0,\ldots,c_d)$ is \emph{unimodal} if $c_0\le\cdots\le c_r\ge\cdots\ge c_d$ for some $r$.
A nonnegative sequence is \emph{log-concave} if $c_j^2\ge c_{j-1}c_{j+1}$ for $1\le j\le d-1$.
A real polynomial is \emph{real-rooted} if all its complex zeros are real. Our first result answers their question and determines exactly when the two stronger properties hold.

\begin{theorem}\label{thm:main-ehrhart}
Let $d,N\ge2$ be integers, and write $L_{d,N}(t)=L_{T_{d,N}}(t)=\sum_{j=0}^d\ell_jt^j$.
The sequence $(\ell_0,\ldots,\ell_d)$ is positive and unimodal. Moreover:
\begin{enumerate}
\item[(i)] It is log-concave if and only if $2\le d\le4$, or $d=5$ and $N\in\{2,3\}$. In each of these cases, $\ell_j^2>\ell_{j-1}\ell_{j+1}$ for $1\le j\le d-1$.
\item[(ii)] $L_{d,N}(t)$ is real-rooted if and only if $d=2$, or $d=3$ and $N\ge35$.
\end{enumerate}
\end{theorem}

We next consider the integer decomposition property. A lattice polytope $P$ has the \emph{integer decomposition property}, abbreviated \emph{IDP}, if for every integer $h\ge1$ and every $x\in hP\cap\Z^d$ there are $x_1,\ldots,x_h\in P\cap\Z^d$ with $x=x_1+\cdots+x_h$.
A \emph{lattice triangulation} of $P$ is a covering by lattice simplices that meet in common faces. 
It is \emph{unimodular} if every maximal simplex has normalized volume one, and \emph{regular} if it
is induced by the lower faces of a lifting of its vertices. 
The precise lattice and lifting conventions are given in \cref{sec:preliminaries}.
A unimodular triangulation implies IDP, but the converse fails even for one-row simplices.

For $1\le q<N$, abbreviate $S_{d;q,N}=S_{(N-q,\ldots,N-q,N)}$.  The case $q=1$ was classified in \cite[Theorem~A]{BCJ}:
$S_{d;1,N}$ is IDP, equivalently admits a unimodular triangulation, precisely when $N\equiv0$ or $1\pmod d$. 
After that, Bruckamp, Caicedo, and Juhnke posed the following open problem.

\begin{question}(\cite[Question~3.14]{BCJ})\label{Question3.14}
For which values of $q$ and $N$, with $1\leq q<N$, does the one-row Hermite normal form simplex $S_{d;q,N}$ admit a unimodular triangulation or satisfy the integer decomposition property?
\end{question}

To state our answer, recall that the \emph{negative continued fraction} of a rational number $x>1$ is the unique expression
\[x=[a_1,\ldots,a_s]^-
=a_1-\cfrac1{a_2-\cfrac1{\ddots-\cfrac1{a_s}}},\qquad a_i\in\Z,\quad a_i\ge2.
\]
It is obtained by taking ceilings in the Euclidean algorithm; \cref{sec:negative-euclid} records the recurrence used here.

\begin{theorem}\label{thm:main-constant}
Let $d,N\ge2$ and $1\le q<N$ be integers, and set $m=d-1$ and $B=N-mq-1$. The following conditions are equivalent:
\begin{enumerate}
\item[(i)] $S_{d;q,N}$ is IDP;
\item[(ii)] $S_{d;q,N}$ admits a unimodular triangulation;
\item[(iii)] $S_{d;q,N}$ admits a regular unimodular triangulation;
\item[(iv)] either $B=0$, or $B>0$ and every entry of
$N/B=[a_1,\ldots,a_s]^-$ satisfies $a_i\equiv2\pmod m$.
\end{enumerate}
\end{theorem}

The geometric equivalence in \cref{thm:main-constant} is a specialization of the two-parameter theory of Dais, Haus, and Henk
\cite[Theorem~5.13]{DHH}; see also \cite[Theorem~7.3]{DHZ}.
In their terminology, a basic triangulation is unimodular and a coherent triangulation is regular.
The continued-fraction mechanism also appears in the more general criterion of Davis, Logvinenko, and Reid
\cite[Section~2.3 and Theorem~2.5]{DLR}. 
Sato \cite[Theorem~3.6]{Sato} (or \cite[Section 5]{KYSato}) states the displayed congruence under the additional hypothesis $\gcd(N,B)=1$.

We express these results directly in the parameters of the Hermite simplex. We give explicit
lattice calculations valid also when $\gcd(N,B)>1$, a regular triangulation in these coordinates, and the remainder test of \cref{prop:remainder-test}.

As a generalization of \cref{Question3.14}, Bruckamp, Caicedo, and Juhnke also posed the following open problem.

\begin{question}(\cite[Question~3.13]{BCJ})\label{Question3.13}
Can one find effective criteria on $\avec\in \mathbb{N}^d$ ensuring that the one-row Hermite normal form simplex $S_{\avec}$ admits a unimodular triangulation or satisfies the integer decomposition property? 
\end{question}

For arbitrary one-row simplices, a single integer-valued function gives an IDP test. Let
$A=\sum_{i=1}^{d-1}a_i$ and set
\begin{equation}\label{eq:height-intro}
H_{\avec}(t)=\sum_{i=1}^{d-1}\ceil{\frac{a_it}{N}}-\floor{\frac{(A-1)t}{N}}, \qquad t\in\{0,\ldots,N-1\},
\end{equation}
where $\floor{x}$ and $\ceil{x}$ are respectively the greatest integer at most $x$ and the least integer at least $x$.
As shown in \cref{lem:height-and-mapping}, $H_{\avec}(t)$ is the height of a distinguished lattice point in the homogenized cone of
$S_{\avec}$; this cone and its height are defined in \cref{sec:cones-and-heights}.

\begin{theorem}\label{thm:main-general}
Let $S_{\avec}$ be as in \eqref{eq:one-row-family}. It is IDP if and only if, for every $1\le t<N$ with
$H_{\avec}(t)\ge2$, there is an integer $s$ satisfying
\[1\le s<t,\qquad H_{\avec}(s)=1,\qquad H_{\avec}(t-s)=H_{\avec}(t)-1.
\]
If this condition fails at $t$, then the point
\[ p_t=\left(\ceil{\frac{a_1t}{N}},\ldots,\ceil{\frac{a_{d-1}t}{N}},t\right)\in H_{\avec}(t)S_{\avec}\cap\Z^d
\]
is not a sum of $H_{\avec}(t)$ lattice points of $S_{\avec}$. The test requires $O(dN+N^2)$ arithmetic operations and
$O(N+d)$ integer storage locations.
\end{theorem}

This is a residue-coordinate formulation of the standard Hilbert-basis criterion for IDP; see \cite[Chapter~2]{BG} and
\cite[Section~1.2.2]{HPPS}. The underlying cyclic lattice and its height enumeration are described in \cite[Section~2.2]{BBCHV}.
Related floor-function criteria were developed by Braun, Davis, and Solus \cite[Theorem~2.3 and Corollary~2.4]{BDS} for the
simplex family described at the start of \cref{sec:residue}. Here the same decomposition principle gives a height-decreasing
test for all one-row parameters, together with a decomposition procedure and the obstruction $p_t$. 

For the triangulation part of \cref{Question3.13}, we apply the standard compatibility and lifting criteria to the lattice points
computed by \cref{thm:main-general}. These give finite searches, not a structural classification of arbitrary parameter vectors.
Related finite formulations already occur in \cite[Theorem~4 and Corollary~5]{FZ}; regularity is characterized
by linear inequalities in \cite[Section~2.3]{DRS}.
The distinction from IDP is necessary: the simplex $S_{(34,31,39)}$ is IDP but has no unimodular triangulation.
In \cref{ex:39} we identify it explicitly with the example in \cite[Example~10]{FZ}, also recorded in \cite[Note~6.2(iii)]{DHZ}.

The paper is organized as follows. 
After fixing the conventions in \cref{sec:preliminaries}, we prove
\cref{thm:main-ehrhart} in \cref{sec:ehrhart} and
\cref{thm:main-general} in \cref{sec:residue}.
\cref{sec:constant} develops the continued-fraction criterion and the
remainder algorithm. \cref{sec:general-triangulations} treats the
finite triangulation tests, the preceding counterexample, and a
construction that reduces the dimension.
Finally, \cref{Section-Finily} contains the concluding remarks.

\section{Preliminaries}\label{sec:preliminaries}

\subsection{Lattice conventions and normalized volume}\label{sec:lattice-conventions}

A \emph{lattice} of rank $r$ in a real vector space is a subgroup
$M=\Z b_1+\cdots+\Z b_r$ generated by linearly independent vectors.
Its real span is denoted $M_{\R}$. The \emph{covolume} of a
full-rank lattice in a Euclidean space is the Euclidean volume of
$\{\sum_{i=1}^r\theta_i b_i:0\le\theta_i<1\}$.
For an affine subspace $U$ containing a lattice point $u_0$, its
\emph{induced affine lattice} is
\[
 U\cap M=u_0+M_U,\qquad M_U=M\cap(U-u_0).
\]
All volumes and unimodularity statements for a polytope in $U$ are
relative to $M_U$, assumed to span $U-u_0$.

Write $\aff(X)$ for the affine hull of a set $X$, and
$\vertices(P)$ for the vertex set of a polytope $P$.
A \emph{face} of $P$ is the set where an affine function attains
its minimum on $P$; we also include the empty face.
A \emph{facet} is a face of codimension one, and a vertex is a
zero-dimensional face. Interior and boundary for a polytope are
taken relative to its affine hull unless stated otherwise.
For an $r$-dimensional lattice polytope $P$, take $U=\aff(P)$.
Its \emph{normalized volume} is $r!$ times its volume measured
so that a fundamental
parallelepiped of $M_U$ has volume one. In particular,
\[
 \nvol_M\bigl(\conv(v_0,\ldots,v_r)\bigr)
 =\left|\det_{M_U}(v_1-v_0,\ldots,v_r-v_0)\right|,
 \qquad U=\aff(v_0,\ldots,v_r),
\]
where the determinant is taken in any basis of $M_U$.
A lattice simplex is \emph{unimodular} if this determinant has
absolute value one. We omit $M$ from $\nvol_M$ when the lattice is
understood. For full-dimensional polytopes in $\R^d$ with lattice
$\Z^d$, this is $d!$ times Euclidean volume.
An \emph{affine unimodular map} is an affine bijection of affine
spaces that maps one induced affine lattice bijectively onto the
other.

We use the row convention for Hermite normal form: an integral
lower triangular matrix $H=(h_{ij})_{1\le i,j\le d}$ has positive
diagonal entries and satisfies $0\le h_{ij}<h_{ii}$ for $j<i$.
The associated simplex is the convex hull of the origin and the rows
of $H$; compare \cite[Section~2.1]{BBCHV}.
The terminology in \eqref{eq:one-row-family} and \eqref{eq:two-row-family}
refers to this convention. Later, homogenized vertex matrices will
instead have vertices as \emph{columns}; these two conventions will
be kept distinct.

For an integer $u$ and a positive integer $N$, write
\[
 [u]_N=u-N\floor{u/N}\in\{0,\ldots,N-1\}
\]
for the least nonnegative residue. For a real number $x$, its
fractional part is $\{x\}=x-\floor{x}$. Empty sums and products
have values zero and one, respectively, and $\conv(\varnothing)
=\varnothing$.

\subsection{Cones}\label{sec:cones-and-heights}

For vectors $u_1,\ldots,u_k$, let
$\cone(u_1,\ldots,u_k)=\{\sum_i\theta_i u_i:\theta_i\ge0\}$.
A cone generated by finitely many lattice vectors is a
\emph{rational polyhedral cone}; it is \emph{pointed} if
$C\cap(-C)=\{0\}$. Its lattice points form an additive
\emph{semigroup}, namely a set containing zero and closed under
addition. For a pointed rational cone $C$ in a lattice $M$, its
\emph{Hilbert basis} $\Hilb_M(C)$ is the set of nonzero elements
of $C\cap M$ that are not sums of two nonzero elements of
$C\cap M$. This is the unique finite minimal generating set of
$C\cap M$; see \cite[Chapter~2]{BG}.
A cone is \emph{simplicial} if its generating rays are linearly
independent. Such a cone is \emph{unimodular} if the primitive
lattice generators of its rays form a basis of
$M\cap\operatorname{span}_{\R}(C)$; a ray generator is primitive
if it is the first nonzero lattice point on that ray.

The \emph{homogenized cone} of a lattice polytope $P\subset\R^d$
is
\[
 C(P)=\cone\{(1,v):v\in\vertices(P)\}
     =\{(h,x):h\ge0,\ x\in hP\}.
\]
Its lattice is $\Z\times\Z^d$, restricted to its linear span
when necessary, and its \emph{height} is the first coordinate $h$.
Every nonzero lattice point of $C(P)$ has positive integral height.
The definition of IDP is therefore equivalent to
\begin{equation}\label{eq:idp-hilbert-criterion}
 P\text{ is IDP}
 \quad\Longleftrightarrow\quad
 \Hilb_{\Z^{d+1}}(C(P))=\{(1,x):x\in P\cap\Z^d\}.
\end{equation}
This standard criterion is also discussed in
\cite[Section~1.2.2]{HPPS}.

For a full-dimensional lattice simplex
$P=\conv(v_0,\ldots,v_d)$, let
\[
 V_P=\bigl((1,v_0)^T\ \cdots\ (1,v_d)^T\bigr),\qquad
 \Lambda_P=V_P^{-1}\Z^{d+1}.
\]
Let $E_0,\ldots,E_d$ be the standard basis of $\R^{d+1}$ and set
$\Delta=\conv(E_0,\ldots,E_d)$.
The map $V_P$ identifies
$\Lambda_P\cap\R_{\ge0}^{d+1}$ with $C(P)\cap\Z^{d+1}$,
and $\Delta\cap\Lambda_P$ with the height-one slice.
In these coordinates the height is
$\htop(\lambda)=\sum_{i=0}^d\lambda_i$.
The \emph{box representatives} are the elements of
\[
 \operatorname{Box}(P)=\Lambda_P\cap[0,1)^{d+1}.
\]
They represent the finite group $\Lambda_P/\Z^{d+1}$, whose order
is $|\det V_P|=\nvol(P)$. Their images form the lattice points of
the half-open fundamental parallelepiped
\[
 \Pi_P=\left\{\sum_{i=0}^d\theta_i(1,v_i):0\le\theta_i<1\right\}.
\]
Every $\lambda\in\Lambda_P\cap\R_{\ge0}^{d+1}$ has the unique
coordinatewise decomposition
\begin{equation}\label{eq:box-decomposition-general}
 \lambda=n+\beta,\qquad
 n\in\Z_{\ge0}^{d+1},\quad \beta\in\operatorname{Box}(P).
\end{equation}
Indeed, $n_i=\floor{\lambda_i}$ and $\beta_i=\{\lambda_i\}$.
Consequently, testing IDP for a simplex reduces to decomposing its
finitely many box representatives into height-one points.
For this lattice description in the one-row case, see
\cite[Section~2.2]{BBCHV}.

\subsection{Triangulations and regular subdivisions}\label{sec:lifting}

A \emph{polyhedral subdivision} of a polytope $P$ is a finite
collection of polytopes, called cells, closed under taking faces,
whose union is $P$ and such that the intersection of any two cells
is a face of each. The empty set is allowed as a face.
A \emph{triangulation} is a subdivision whose cells are simplices.
For a finite set $\mathcal A$ of points, a triangulation of
$\mathcal A$ means a triangulation of $\conv(\mathcal A)$ with
vertices in $\mathcal A$; it is not required to use every point.
It is \emph{full} if every point of $\mathcal A$ is a vertex.
A lattice triangulation has lattice vertices and is unimodular if
all its maximal simplices are unimodular. A subdivision
$\mathcal S'$ \emph{refines} $\mathcal S$ if every cell of
$\mathcal S'$ is contained in a cell of $\mathcal S$.
These conventions agree with \cite[Section~1.1]{HPPS}.

A \emph{lifting function} on $\mathcal A$ is a function
$w:\mathcal A\to\R$. An affine function $\ell$ on
$\aff(\mathcal A)$ is a \emph{lower supporting function} if
$\ell(p)\le w(p)$ for all $p\in\mathcal A$.
The nonempty sets
\[
 \conv\{p\in\mathcal A:\ell(p)=w(p)\}
\]
are the projections of the lower faces of
$\conv\{(p,w(p)):p\in\mathcal A\}$ and form the
\emph{regular subdivision} induced by $w$.
A triangulation is \emph{regular}, or \emph{coherent}, if it is
induced in this way, with the equality sets of the maximal lower
supporting functions equal to the vertex sets of the maximal
simplices. Thus, for each such simplex $\sigma$, its supporting
function satisfies
\begin{equation}\label{eq:strict-lifting}
 \ell_\sigma(p)=w(p)\quad(p\in\vertices(\sigma)),\qquad
 \ell_\sigma(q)<w(q)\quad(q\in\mathcal A\setminus\vertices(\sigma)).
\end{equation}
See \cite[Section~1.1, Equations~(1.1)--(1.2)]{HPPS}.

We record three standard facts in the form needed below.

\begin{lemma}[{\cite[Sections~1.1 and~1.2.2]{HPPS}}]\label{lem:unimodular-idp}
A lattice polytope admitting a unimodular triangulation is IDP.
A unimodular simplex contains no lattice points other than its
vertices. In particular, every unimodular triangulation of a lattice
polytope uses all its lattice points.
\end{lemma}

\begin{lemma}[{\cite[Lemma~2.3.15\textup{(1), (3)}]{DRS}}]\label{lem:small-perturbation}
Let $\mathcal A$ be a finite point configuration and let
$w_0\in\R^{\mathcal A}$ be a lifting function. There is a
neighborhood of $w_0$ in which every lifting induces a subdivision
refining the subdivision induced by $w_0$. Outside a finite union
of proper linear hyperplanes in $\R^{\mathcal A}$, lifting
functions induce triangulations. Hence a sufficiently small generic
perturbation of $w_0$ induces a regular triangulation refining its
subdivision.
\end{lemma}

Here $\R^{\mathcal A}$ is the vector space of real-valued functions
on $\mathcal A$; ``generic'' refers to avoiding the exceptional
hyperplanes in the lemma.
The next statement combines the planar emptiness criterion with
regular refinement; see \cite[Section~1.1, Proposition~1.1 and
Lemma~2.1]{HPPS}.

\begin{lemma}\label{lem:polygon}
Every triangulation of a lattice polygon using all its lattice
points is unimodular. Moreover, every lattice polygon admits a
regular triangulation using all its lattice points, and hence a
regular unimodular triangulation.
\end{lemma}

\subsection{Ehrhart series and box heights}\label{sec:ehrhart-conventions}

For a lattice polytope $P$ of dimension $r$, its \emph{Ehrhart
series} and \emph{$h^*$-polynomial} are defined by
\[
 \sum_{h\ge0}L_P(h)z^h=\frac{h_P^*(z)}{(1-z)^{r+1}},
 \qquad h_P^*(z)=\sum_{j=0}^r h_j^*z^j.
\]
Equivalently,
\begin{equation}\label{eq:binomial-transform}
 L_P(t)=\sum_{j=0}^r h_j^*\binom{t+r-j}{r},
 \qquad
 \binom{x}{r}=\frac{x(x-1)\cdots(x-r+1)}{r!}.
\end{equation}
These are polynomial identities, with $\binom{x}{0}=1$;
see \cite[Chapter~3]{BR}. For a full-dimensional lattice simplex,
\eqref{eq:box-decomposition-general} gives the familiar formula
\begin{equation}\label{eq:hstar-box-general}
 h_P^*(z)=\sum_{\beta\in\operatorname{Box}(P)}z^{\htop(\beta)};
\end{equation}
compare \cite[Section~2.2]{BBCHV}.
The height of a box representative is also called its \emph{age}.
It lies in $\{0,\ldots,r\}$, and only the zero representative has
height zero. The coefficients $h_j^*$ should be distinguished from
the coefficients of $L_P(t)$ in the ordinary power basis, which
are the subject of \cref{sec:ehrhart}.

\section{Ehrhart coefficients of the two-row family}\label{sec:ehrhart}

Our starting point is the formula of Bruckamp, Caicedo, and Juhnke
\cite[Theorems~5.10 and~5.12]{BCJ}:
\begin{equation}\label{eq:hstar-two-row}
 h_{T_{d,N}}^*(z)=1+N\sum_{j=1}^{d-2}z^j+(N-1)z^{d-1},
\end{equation}
and hence
\begin{equation}\label{eq:ehrhart-formula}
L_{d,N}(t)=\binom{t+d}{d}-\binom{t+1}{d}+N\left(\binom{t+d}{d+1}-\binom{t+1}{d+1}\right).
\end{equation}
The second identity follows from \eqref{eq:binomial-transform} by
telescoping Pascal's identity: $\binom{x}{d}=\binom{x+1}{d+1}-\binom{x}{d+1}$.

\subsection{Coefficient ratios}

Let $P(t) = \sum_i p_i t^i$ and $Q(t) = \sum_i q_i t^i$ be polynomials with nonnegative coefficients. The support of $P$ is $\{i:p_i\ne0\}$, and similarly for $Q$. We extend both coefficient sequences by zero outside their supports and define
\begin{equation}\label{eq:coefficient-order}
P \preceq Q \quad \Longleftrightarrow \quad p_i q_j - p_j q_i \ge 0 \quad \text{for all integers } i < j.
\end{equation}
For indices with $p_i,p_j>0$, the corresponding inequality is equivalent to $q_i/p_i\le q_j/p_j$. We use only the definition and the following elementary consequences.

\begin{lemma}\label{lem:order-multiplication}
Suppose that $P \preceq Q$ and let $c > 0$. Then $(t+c)P \preceq (t+c)Q$. Consequently, multiplication by any product of linear factors $(t+c)$ with $c > 0$ preserves the relation $\preceq$.
\end{lemma}
\begin{proof}
Let $p'_i$ and $q'_i$ denote the coefficients of the polynomials $(t+c)P$ and $(t+c)Q$, respectively. Then $p'_i = c p_i + p_{i-1}$ and $q'_i = c q_i + q_{i-1}$. For $i < j$, we have
\begin{align*}
p'_i q'_j - p'_j q'_i &= (c p_i + p_{i-1})(c q_j + q_{j-1}) - (c p_j + p_{j-1})(c q_i + q_{i-1}) \\
&= c^2(p_i q_j - p_j q_i) + c(p_i q_{j-1} - p_{j-1} q_i) + c(p_{i-1} q_j - p_j q_{i-1}) + (p_{i-1} q_{j-1} - p_{j-1} q_{i-1}).
\end{align*}
By the assumption \eqref{eq:coefficient-order}, the terms in the first, third, and fourth parentheses are nonnegative. The second term, $p_i q_{j-1} - p_{j-1} q_i$, vanishes when $i = j-1$, and is nonnegative by \eqref{eq:coefficient-order} whenever $i < j-1$. Therefore, $p'_i q'_j - p'_j q'_i \ge 0$, which proves the assertion for one factor. Induction gives the result for a product.
\end{proof}

\begin{lemma}\label{lem:difference-selection}
Suppose that $P(t) = \sum_{j=0}^r p_j t^j$ and $Q(t) = \sum_{j=0}^{r+1} q_j t^j$ are polynomials with positive coefficients, satisfying the condition
\[ P \preceq Q \preceq tP.\]
Let $(c_0, \ldots, c_r)$ be a real sequence such that for each $1 \le j \le r$, the forward difference $c_j - c_{j-1}$ is chosen from the set $\{p_j - p_{j-1}, q_j - q_{j-1}\}$. Then the sequence $(c_0, \ldots, c_r)$ is unimodal.
\end{lemma}
\begin{proof}
By applying the relation $P \preceq Q$ to consecutive indices, we obtain 
\[\frac{p_j}{p_{j-1}} \le \frac{q_j}{q_{j-1}} \qquad (1 \le j \le r).
\]
Similarly, since the coefficient of $t^j$ in $tP$ is $p_{j-1}$, evaluating the relation $Q \preceq tP$ at consecutive indices yields
\[\frac{q_j}{q_{j-1}} \le \frac{p_{j-1}}{p_{j-2}} \qquad (2 \le j \le r+1).
\]
These inequalities show that the ratio $p_j/p_{j-1}$ is nonincreasing in $j$. Hence the positive sequence $(p_j)_{j=0}^r$ is log-concave, and therefore unimodal. Thus, there exists an index $k \in \{0, \ldots, r\}$ such that
\[ p_0 \le \cdots \le p_k \ge p_{k+1} \ge \cdots \ge p_r.
\]

For $1 \le j \le k$, the monotonicity up to $k$ ensures $p_j/p_{j-1} \ge 1$. The first ratio inequality then implies $q_j/q_{j-1} \ge p_j/p_{j-1} \ge 1$, so $q_j - q_{j-1} \ge 0$. Hence, both candidate values for the difference $c_j - c_{j-1}$ are nonnegative in this range.

Conversely, for $k+2 \le j \le r$, the descent past $k$ means $p_{j-1}/p_{j-2} \le 1$. The second ratio inequality implies $q_j/q_{j-1} \le p_{j-1}/p_{j-2} \le 1$, so both candidate differences are nonpositive.

If $k=r$, all differences $c_j-c_{j-1}$ are nonnegative and the claim follows. Otherwise, only the sign of $c_{k+1}-c_k$ remains unrestricted. If $c_{k+1} - c_k \ge 0$, the sequence $(c_j)_{j=0}^r$ reaches its peak at $k+1$; if $c_{k+1} - c_k \le 0$, the peak occurs at $k$. Thus the sequence is unimodal. (For $k=0$, the initial range of nonnegative differences is empty.)
\end{proof}

\subsection{Proof of unimodality}

\begin{proof}[Proof of the unimodality assertion in \cref{thm:main-ehrhart}]
For $d=2$, we have 
\[L_{2,N}(t) = 1 + \frac{N+2}{2}t + \frac{N}{2}t^2.
\]
The middle coefficient is at least as large as either adjacent coefficient.
Suppose henceforth that $d \ge 3$. Let 
\[G(t) = \prod_{s=1}^{d-2}(t+s) \qquad \text{and} \qquad \kappa = N(d-1)-(d+1).
\]
The hypotheses give $\kappa \ge d-3 \ge 0$. Define the polynomials $U(t)$ and $V(t)$ by
\begin{align*}
    U(t) &= d(d^2-1) + \bigl(N(d-1)^2+2d(d+1)\bigr)t + \bigl(N(d-1)+2(d+1)\bigr)t^2, \\
    V(t) &= d(d^2-1) + (N+2)(d^2-1)t + N(3d-1)t^2 + 2Nt^3.
\end{align*}
Both $U(t)$ and $V(t)$ have positive coefficients on their supports. We first prove
\begin{equation}\label{eq:short-polynomial-order}
 U \preceq V \preceq tU.
\end{equation}
Write the coefficients of $U$ by $u_0, u_1, u_2$ and those of $V$ by $u_0, v_1, v_2, v_3$. Since the coefficients on these supports are positive, \eqref{eq:short-polynomial-order} is equivalent to
\[v_1 \ge u_1, \qquad u_1v_2 \ge u_2v_1, \qquad u_1v_1 \ge u_0v_2, \qquad u_2v_2 \ge u_1v_3.
\]
For the first two inequalities, we have
$v_1 - u_1 = 2\kappa$ and $v_2 - u_2 = 2(N+\kappa)$. Thus
\[u_1v_2 - u_2v_1 = 2Nu_1 + 2\kappa(u_1-u_2) > 0, \qquad \text{where} \quad u_1 - u_2 = (d-1)\bigl(N(d-2)+2(d+1)\bigr) > 0.
\]
For the remaining inequalities, we obtain 
\begin{align*}
u_1v_1 - u_0v_2 &= (d^2-1)\bigl(N(N+1)(d-1)^2 + (N+4d)(d+1)\bigr) > 0, \\
u_2v_2 - u_1v_3 &= N(N+2)(d^2-1) > 0.
\end{align*}
This confirms the relation \eqref{eq:short-polynomial-order}. By \cref{lem:order-multiplication}, the polynomials
$P(t) = G(t)U(t)$ and $Q(t) = G(t)V(t)$ satisfy
\begin{equation}\label{eq:long-polynomial-order}
P \preceq Q \preceq tP.
\end{equation}
Note that $\deg P(t) = d$ and $\deg Q(t) = d+1$, and both polynomials have positive coefficients on their supports.

Next, we express the differences of the Ehrhart coefficients in terms of $P(t)$ and $Q(t)$. Define the auxiliary polynomial 
$R(t) = t(\kappa+Nt)G(t)$.
Using the coefficient identities for $U(t)$ and $V(t)$, along with the fact that $v_3 = 2N$, we obtain
\begin{equation}\label{eq:PQR}
Q(t) = P(t) + 2(1+t)R(t).
\end{equation}
By expanding the binomial factors in \eqref{eq:ehrhart-formula}, we get
\begin{align*}
 (d+1)! L_{d,N}(t) = (Nt+d+1)\prod_{s=1}^d(t+s) - (Nt-\kappa)(t+1)t\prod_{s=1}^{d-2}(t-s).
\end{align*}
To verify the factorization of the second term, we combine the two subtracted components from \eqref{eq:ehrhart-formula} to obtain
\[(d+1)(t+1)t\prod_{s=1}^{d-2}(t-s) + N(t-d+1)(t+1)t\prod_{s=1}^{d-2}(t-s),
\]
which gives the factor $Nt - N(d-1) + (d+1) = Nt - \kappa$. According to the definition of $U(t)$, we obtain
\[U(t) + (1+t)t(\kappa+Nt) = (Nt+d+1)(t+d-1)(t+d).
\]
By replacing the argument of $R(t)$ by its negative, we have
\[(-1)^d R(-t) = (Nt-\kappa)t\prod_{s=1}^{d-2}(t-s).
\]
Therefore, we have 
\begin{equation}\label{eq:parity-ehrhart}
(d+1)! L_{d,N}(t) = P(t) + (1+t)\bigl(R(t) - (-1)^d R(-t)\bigr).
\end{equation}

Let $P(t) = \sum p_j t^j$, $Q(t) = \sum q_j t^j$, and $R(t) = \sum r_j t^j,\ r_j=0\ (j<0)$, and define the scaled coefficients $c_j = (d+1)! \ell_j$. Since the coefficient of $t^j$ in $R(t) - (-1)^d R(-t)$ vanishes when $j \equiv d \pmod{2}$ and equals $2r_j$ otherwise, equation \eqref{eq:parity-ehrhart} yields the piecewise formula
\[
    c_j = \begin{cases}
    p_j + 2r_{j-1}, & j \equiv d \pmod{2}, \\
    p_j + 2r_j, & j \not\equiv d \pmod{2}.
    \end{cases}
\]
Taking the first backward difference, $c_j - c_{j-1}$, we immediately obtain $p_j - p_{j-1}$ for the parity $j \equiv d \pmod{2}$. Conversely, for the alternate parity, we have
\[
    c_j - c_{j-1} = p_j - p_{j-1} + 2(r_j - r_{j-2}).
\]
Taking coefficient differences in \eqref{eq:PQR} identifies this expression with $q_j - q_{j-1}$. Hence
\begin{equation}\label{eq:alternating-differences}
    c_j - c_{j-1} =
    \begin{cases}
    p_j - p_{j-1}, & j \equiv d \pmod{2}, \\
    q_j - q_{j-1}, & j \not\equiv d \pmod{2},
    \end{cases}
    \qquad 1 \le j \le d.
\end{equation}
By applying \cref{lem:difference-selection} to the relation \eqref{eq:long-polynomial-order} with the difference structure in \eqref{eq:alternating-differences}, we conclude that the sequence $(c_0, \ldots, c_d)$ is unimodal. 
Scaling by the positive constant $1/(d+1)!$ implies that the original Ehrhart coefficient sequence $(\ell_j)_{j=0}^d$ is also unimodal. 
Finally, $P(t)$ has positive coefficients and $R(t)$ has nonnegative coefficients. 
Thus the displayed formula for $c_j$ also proves that every $c_j$ is positive.
\end{proof}

The proof just given uses $N\ge2$ but not integrality of $N$. 
Consequently, the polynomial defined by \eqref{eq:ehrhart-formula} has positive unimodal coefficients for every real $N\ge2$. 
Only integer parameters are interpreted as the lattice simplices $T_{d,N}$.

\subsection{The three leading coefficients}

\begin{lemma}
Suppose that $d \ge 3$. If we write $d! L_{d,N}(t) = \sum_{j=0}^d c_j t^j$, then the top three coefficients are given by
\begin{equation}\label{eq:leading-coefficients}
c_d = N(d-1), \quad c_{d-1}= \frac{d(d-1)}{2}(N+2), \quad  c_{d-2}= \frac{d(d-1)^2}{24} \bigl(N(d-1)(d-2)+24\bigr).
\end{equation}
\end{lemma}
\begin{proof}
For a multiset $S=(x_1,\ldots,x_r)$, set
\[e_j(S)=\sum_{1\le i_1<\cdots<i_j\le r}x_{i_1}\cdots x_{i_j},
 \qquad s_j(S)=\sum_{i=1}^r x_i^j,
\]
with $e_0(S)=1$ and $e_j(S)=0$ for $j>r$.
Let $[t^j]F(t)$ be the coefficient of $t^j$ in a polynomial $F$.
We use the first three Newton identities
\[
 e_1=s_1,\qquad e_2=\frac{s_1^2-s_2}{2},\qquad
 e_3=\frac{s_1^3-3s_1s_2+2s_3}{6};
\]
see \cite[Chapter~7]{RP.Stanley99}.
Define the sets
\[X = \{0,1,\ldots,d\}, \quad Y = \{1,0,-1,\ldots,-(d-1)\}, \quad Z = \{1,0,-1,\ldots,-(d-2)\}.
\]
Therefore, we obtain 
\[
\renewcommand{\arraystretch}{2.2}
\begin{array}{c|ccc}
  & s_1 & s_2 & s_3 \\ \hline
X & \frac{d(d+1)}{2} & \frac{d(d+1)(2d+1)}{6} & \frac{d^2(d+1)^2}{4} \\
Y & 1-\frac{d(d-1)}{2} & 1+\frac{d(d-1)(2d-1)}{6} & 1-\frac{d^2(d-1)^2}{4} \\
Z & 1-\frac{(d-1)(d-2)}{2} & 1+\frac{(d-1)(d-2)(2d-3)}{6} & \text{not needed}
\end{array}
\]
Let $f(t)$ and $g(t)$ denote the two difference expressions appearing in \eqref{eq:ehrhart-formula}, i.e., $L_{d,N}(t) = f(t) + N g(t)$. 

The polynomial $d! f(t)$ is the difference between $\prod_{x=1}^d (t+x)$ and $\prod_{z \in Z} (t+z)$. Consequently, its leading term of degree $d$ vanishes. Since the inclusion of zero in $X$ does not alter the elementary symmetric polynomials of positive degree, coefficient extraction yields
\begin{align*}
    [t^{d-1}]\,d! f(t) &= s_1(X) - s_1(Z) = d(d-1), \\
    [t^{d-2}]\,d! f(t) &= \frac{s_1(X)^2 - s_1(Z)^2 - s_2(X) + s_2(Z)}{2} = d(d-1)^2.
\end{align*}

Similarly, the polynomial $(d+1)! g(t)$ is given by the difference between the product $\prod_{x \in X} (t+x)$ and the product $\prod_{y \in Y} (t+y)$. Their terms of degree $d+1$ cancel upon subtraction. 
By substituting the tabulated power sums into the three symmetric-sum identities, we obtain
\begin{align*}
e_1(X) - e_1(Y) &= (d+1)(d-1), \qquad e_2(X) - e_2(Y) = \frac{d(d+1)(d-1)}{2}, \\
e_3(X) - e_3(Y) &= \frac{d(d+1)(d-1)^3(d-2)}{24}.
\end{align*}
The second equality also follows from $s_1(X) + s_1(Y) = d+1$ and $s_2(X) - s_2(Y) = d^2-1$, which give $e_2(X) - e_2(Y) = \frac{1}{2}\bigl((d^2-1)(d+1) - (d^2-1)\bigr)$. For the third equality, direct substitution yields:
\begin{align*}
6\bigl(e_3(X) - e_3(Y)\bigr)&= \left(\frac{d(d+1)}{2}\right)^3 - \left(1-\frac{d(d-1)}{2}\right)^3 - \frac{d^2(d+1)^2(2d+1)}{4} \\
&+ 3\left(1-\frac{d(d-1)}{2}\right)\left(1+\frac{d(d-1)(2d-1)}{6}\right) + \frac{d^2(d+1)^2+d^2(d-1)^2}{2} - 2 \\
&= \frac{d(d+1)(d-1)^3(d-2)}{4}.
\end{align*}

Dividing these three coefficient differences by $d+1$ provides the leading coefficients of $d! g(t)$. 
After multiplying by $N$ and summing the coefficients of $d! f(t)$, we get \eqref{eq:leading-coefficients}.
\end{proof}

\subsection{Classification of log-concavity}

\begin{proof}[Proof of \cref{thm:main-ehrhart}\textup{(i)}]
Since scaling a polynomial by a positive constant preserves the log-concavity of its coefficient sequence, we may work directly with the scaled coefficients $c_j$. By \eqref{eq:leading-coefficients}, for any dimension $d \ge 3$, the discriminant of the top three coefficients evaluates to
\begin{align*}
c_{d-1}^2 - c_{d-2}c_d &= \frac{d(d-1)^2}{24} \left( 6d(N+2)^2 - N(d-1)\bigl(N(d-1)(d-2)+24\bigr) \right) \\
    &= -\frac{d(d-1)^2}{24} \left( (d^3-4d^2-d-2)N^2 - 24N - 24d \right).
\end{align*}

Suppose that $d \ge 6$. Because $N \ge 2$, we can bound the linear terms by $24N \le 12N^2$ and $24d \le 6dN^2$. Applying these bounds, the expression within the brackets in the last line is bounded below by $N^2(d^3-4d^2-7d-14)$.
For $d \ge 6$, we further have
\[ d^3-4d^2-7d-14 \ge 2d^2-7d-14 \ge 16.
\]
Indeed, $2d^2-7d-14$ equals $16$ at $d=6$ and has positive forward difference $4d-5$ for $d\ge6$. Consequently, the bracketed expression is positive, yielding $c_{d-1}^2 < c_{d-2}c_d$ for every dimension $d \ge 6$. Thus, the sequence is not log-concave for $d \ge 6$.

For the low-dimensional cases $d \in \{2, 3, 4, 5\}$, we verify the log-concavity assertions by direct computation. Equation \eqref{eq:ehrhart-formula} yields the following scaled Ehrhart polynomials:
\begin{align*}
2L_{2,N}(t) &= 2 + (N+2)t + Nt^2, \\
6L_{3,N}(t) &= 6 + (N+12)t + 3(N+2)t^2 + 2Nt^3, \\
8L_{4,N}(t) &= 8 + 2(N+8)t + 3(N+4)t^2 + 2(N+2)t^3 + Nt^4.
\end{align*}
For $d=2$, the sole internal log-concavity discriminant is $(N+2)^2 - 2N = N^2+2N+4$. For $d=3$, the two discriminants are
\begin{align*}
(N+12)^2 - 18(N+2) &= N^2+6N+108, \\
9(N+2)^2 - 2N(N+12) &= 7N^2+12N+36.
\end{align*}
For $d=4$, the corresponding differences are
\begin{align*}
4(N+8)^2 - 24(N+4) &= 4(N^2+10N+40), \\
9(N+4)^2 - 4(N+8)(N+2) &= 5N^2+32N+80, \\
4(N+2)^2 - 3N(N+4) &= N^2+4N+16.
\end{align*}
All these expressions in $N$ are positive for $N \ge 2$.

Finally, for dimension $d=5$, the Ehrhart polynomial takes the form
\begin{align*}
L_{5,N}(t) &= 1 + \left(\frac{2N}{15}+\frac{7}{3}\right)t + \left(\frac{5N}{12}+\frac{11}{6}\right)t^2 + \frac{N+2}{3}t^3 + \frac{N+2}{12}t^4 + \frac{N}{30}t^5.
\end{align*}
By evaluating the four internal discriminants $\ell_j^2 - \ell_{j-1}\ell_{j+1}$, we obtain 
\begin{align*}
 \ell_1^2 - \ell_0\ell_2 &= \frac{16N^2+185N+3250}{900},\quad  \ell_2^2 - \ell_1\ell_3 = \frac{93N^2+476N+1300}{720}, \\
    \ell_3^2 - \ell_2\ell_4 &= \frac{(N+2)(11N+10)}{144}, \qquad \ell_4^2 - \ell_3\ell_5 = \frac{(N+2)(10-3N)}{720}.
\end{align*}
The first three discriminants are positive for all $N \ge 2$. However, the fourth discriminant is non-negative if and only if $N \le \frac{10}{3}$. For integers $N \ge 2$, this condition is satisfied precisely when $N \in \{2, 3\}$, in which case the difference is positive. This completes the classification of the log-concave cases and establishes all claimed strict inequalities.
\end{proof}

\subsection{Classification of real-rootedness}

We use the following instance of Newton's inequalities; see
\cite[Theorem~5.12]{RP.StanleyAC}.

\begin{lemma}\label{lem:real-root-obstruction}
If a real polynomial $C(t)=c_dt^d+c_{d-1}t^{d-1}+c_{d-2}t^{d-2}+\cdots$ of degree $d\ge2$ has only real roots, then we have
\[c_{d-1}^2\ge\frac{2d}{d-1}c_{d-2}c_d.
\]
\end{lemma}

We now complete the proof of \cref{thm:main-ehrhart}\textup{(ii)}.

\begin{proof}[Proof of \cref{thm:main-ehrhart}\textup{(ii)}]
Suppose that $d \ge 4$. 
By substituting the explicit expressions from \eqref{eq:leading-coefficients}, we obtain 
\begin{align*}
c_{d-1}^2 - \frac{2d}{d-1}c_{d-2}c_d&= \frac{d^2(d-1)^2}{12} \left( 3(N+2)^2 - N\bigl(N(d-1)(d-2)+24\bigr) \right) \\
&= -\frac{d^2(d-1)^2}{12} \left( (d^2-3d-1)N^2 + 12N - 12 \right) < 0.
\end{align*}
The expression is negative because $d \ge 4$ implies $d^2-3d-1 \ge 3$. Combined with the hypothesis $N \ge 2$, the quadratic expression in $N$ within the brackets is positive. 
Consequently, by \cref{lem:real-root-obstruction}, we shows that the Ehrhart polynomial is not real-rooted for $d \ge 4$.

In the remaining dimensions, we factor the polynomial. In dimension $d=2$, the polynomial factors as
\[L_{2,N}(t) = \frac{1}{2}(t+1)(Nt+2),
\]
which has real roots for all $N$. In dimension $d=3$, we obtain the factorization
\[L_{3,N}(t) = \frac{1}{6}(t+1)\bigl(2Nt^2+(N+6)t+6\bigr).
\]
The discriminant of the quadratic factor is given by
\[(N+6)^2 - 48N = N^2 - 36N + 36 = (N-18)^2 - 288.
\]
This discriminant is non-negative if and only if $N \le 18-12\sqrt{2}$ or $N \ge 18+12\sqrt{2}$. Since $18-12\sqrt{2} \approx 1.029$, the first interval contains no admissible integer $N \ge 2$. Furthermore, bounding the square root yields $34 < 18+12\sqrt{2} < 35$. For integral $N \ge 2$, we conclude that the quadratic factor (and thus $L_{3,N}(t)$) is real-rooted if and only if $N \ge 35$.
At $N=35$, the factorization is
\[L_{3,35}(t) = \frac{1}{6}(t+1)(7t+2)(10t+3).
\]
This proves \cref{thm:main-ehrhart}.
\end{proof}

\section{A residue criterion for one-row simplices}\label{sec:residue}

We specialize the box description of \cref{sec:cones-and-heights}
to $S_{\avec}$ and prove \cref{thm:main-general}. The coordinate
lattice and height enumeration are standard; compare
\cite[Section~2.2]{BBCHV}. The point of the following calculation
is to express the Hilbert-basis criterion
\eqref{eq:idp-hilbert-criterion} as a test on the single residue
parameter $t$.
For comparison, Braun, Davis, and Solus
\cite[Theorem~2.3]{BDS} give a floor-function test for
\[
 \conv\left(e_1,\ldots,e_n,-\sum_{i=1}^n q_i e_i\right),
 \qquad q_i\in\Z_{>0},\quad q_j\mid1+\sum_{i=1}^n q_i
 \quad(1\le j\le n).
\]
Their proof subtracts a height-one box point from each higher
box point. We use this same principle directly in the lattice
\eqref{eq:general-coordinate-lattice}, without the displayed
divisibility hypotheses.

\subsection{The coordinate lattice}

Set $v_0=0$, $v_i=e_i$ for $1\le i\le d-1$, and
$v_d=\avec$. Let $V=V_{S_{\avec}}$ be the homogenized vertex
matrix from \cref{sec:cones-and-heights}. Then $|\det V|=N$, so
$\nvol(S_{\avec})=N$. Set $\Lambda_{\avec}=V^{-1}\Z^{d+1}$.
By solving $V\lambda=(h,x_1,\ldots,x_d)^T$, we obtain
\begin{equation}\label{eq:inverse-vertex-map}
\lambda_d=\frac{x_d}{N},\quad 
 \lambda_i=x_i-\frac{a_ix_d}{N}\quad(1\le i\le d-1),\quad 
 \lambda_0=h-\sum_{i=1}^{d-1}x_i+\frac{(A-1)x_d}{N}, \quad A=\sum_{i=1}^{d-1}a_i.
\end{equation}
Consequently,
\begin{equation}\label{eq:general-coordinate-lattice}
 \Lambda_{\avec}=\Z^{d+1}
 +\Z\frac1N(A-1,-a_1,\ldots,-a_{d-1},1).
\end{equation}
Its quotient by $\Z^{d+1}$ is cyclic of order $N$, and its
covolume is $1/N$. According to the residue notation from
\cref{sec:lattice-conventions}, the box representatives are
\begin{equation}\label{eq:box-representatives}
 \lambda(t)=\frac1N\bigl([(A-1)t]_N,[-a_1t]_N,\ldots,
 [-a_{d-1}t]_N,t\bigr),\qquad 0\le t<N.
\end{equation}
The last coordinate distinguishes these $N$ representatives.

\begin{lemma}\label{lem:height-and-mapping}
For $0\le t<N$, the representative $\lambda(t)$ has height
$H_{\avec}(t)$ as defined in \eqref{eq:height-intro}, and
\[
 V\lambda(t)=(H_{\avec}(t),p_t),\qquad
 p_t=\left(\left\lceil\frac{a_1t}{N}\right\rceil,\ldots,
 \left\lceil\frac{a_{d-1}t}{N}\right\rceil,t\right).
\]
For $1\le t<N$, one has $1\le H_{\avec}(t)\le d$.
\end{lemma}
\begin{proof}
The identities
\[
 [-a_it]_N=N\left\lceil\frac{a_it}{N}\right\rceil-a_it,
 \qquad [(A-1)t]_N=(A-1)t-N\left\lfloor\frac{(A-1)t}{N}\right\rfloor
\]
show that the coordinate sum of \eqref{eq:box-representatives}
is $H_{\avec}(t)$. Its image under $V$ has $i$-th spatial
coordinate $([-a_it]_N+a_it)/N$ and last coordinate $t$,
giving the stated formula. For $t>0$, the height is a positive
integer less than $d+1$, since $\lambda(t)$ is nonzero and all
its coordinates belong to $[0,1)$.
\end{proof}

\begin{lemma}\label{lem:all-lattice-points}
Let
\[
 J_{\avec}=\{t\in\Z:1\le t<N,\ H_{\avec}(t)=1\}
\]
be the set of \emph{junior indices}. Then
\begin{equation}\label{eq:all-lattice-points}
 S_{\avec}\cap\Z^d=\{v_0,\ldots,v_d\}
 \sqcup\{p_t:t\in J_{\avec}\}.
\end{equation}
Every $\lambda\in\Lambda_{\avec}\cap\R_{\ge0}^{d+1}$ has a
unique expression
\begin{equation}\label{eq:integer-box-decomposition}
 \lambda=n+\lambda(t),\qquad
 n\in\Z_{\ge0}^{d+1},\quad 0\le t<N.
\end{equation}
\end{lemma}
\begin{proof}
The last assertion is \eqref{eq:box-decomposition-general} in the
coordinates \eqref{eq:box-representatives}. At height one, either
$t=0$ and $n=E_i$ for some $i$, or $n=0$ and
$H_{\avec}(t)=1$. By applying $V$, we obtain \eqref{eq:all-lattice-points}. The union is disjoint because
$p_t$ has last spatial coordinate $t\in\{1,\ldots,N-1\}$,
whereas every vertex has last coordinate $0$ or $N$.
Distinct indices also give distinct points.
\end{proof}

In particular, the standard box formula
\eqref{eq:hstar-box-general} becomes
\[
 h_{S_{\avec}}^*(z)=\sum_{t=0}^{N-1}z^{H_{\avec}(t)};
\]
see again \cite[Section~2.2]{BBCHV}.

\subsection{The height-decreasing criterion}

\begin{proof}[Proof of \cref{thm:main-general}]
By \eqref{eq:integer-box-decomposition}, $S_{\avec}$ is IDP
if and only if every $\lambda(t)$ is a sum of height-one
nonnegative lattice points.

Suppose first that $S_{\avec}$ is IDP and
$H_{\avec}(t)=h\ge2$. No summand in a height-one decomposition
of $\lambda(t)$ can be an $E_i$, since every coordinate of
$\lambda(t)$ is less than one. Thus
\[
 \lambda(t)=\lambda(s_1)+\cdots+\lambda(s_h),
 \qquad s_i\in J_{\avec}.
\]
The last coordinate gives $t=s_1+\cdots+s_h$. Thus we have $1\le s_1<t$. The difference $\lambda(t)-\lambda(s_1)$
belongs to $[0,1)^{d+1}$ and represents residue $t-s_1$.
By uniqueness of box representatives, it equals
$\lambda(t-s_1)$. Taking heights gives
$H_{\avec}(t-s_1)=H_{\avec}(t)-1$.

Conversely, suppose the condition in \cref{thm:main-general}
holds. If $1\le s<t<N$, reduction of each coordinate modulo
one gives
\[
 \lambda(s)+\lambda(t-s)-\lambda(t)\in\{0,1\}^{d+1}.
\]
If $H_{\avec}(s)=1$ and
$H_{\avec}(t-s)=H_{\avec}(t)-1$, this nonnegative vector has
coordinate sum zero. Hence
\begin{equation}\label{eq:no-carry-decomposition}
 \lambda(t)=\lambda(s)+\lambda(t-s).
\end{equation}
Induction on the height now decomposes every representative
into height-one points. Together with
\eqref{eq:integer-box-decomposition}, this proves IDP.

If the condition fails at $t$, then
$p_t\in H_{\avec}(t)S_{\avec}\cap\Z^d$ by
\cref{lem:height-and-mapping}. A decomposition into
$H_{\avec}(t)$ points of $S_{\avec}\cap\Z^d$ would lift
to a decomposition of $\lambda(t)$ and, by the necessity
argument, give an admissible index $s$. Thus $p_t$ is the
claimed certificate of failure.

Computing the height array takes $O(dN)$ arithmetic operations.
For each $t$ with $H_{\avec}(t)\ge2$, search the indices
$s\in J_{\avec}$ with $s<t$ for
$H_{\avec}(t-s)=H_{\avec}(t)-1$. This takes at most
$N|J_{\avec}|\le N(N-1)$ comparisons and subtractions.
The heights, junior indices, and one chosen predecessor for
each successful $t$ require $O(N)$ integer storage slots;
the input requires $O(d)$. These are bounds on arithmetic
operations and stored integers, rather than bit-complexity
bounds.
\end{proof}

\begin{corollary}\label{cor:necessary-idp}
If $S_{\avec}$ is IDP, then
\begin{equation}\label{eq:first-height-condition}
 1+\sum_{i=1}^{d-1}[-a_i]_N\le N.
\end{equation}
If every $a_i>0$, this is equivalent to
$\sum_{i=1}^{d-1}a_i\ge(d-2)N+1$.
\end{corollary}
\begin{proof}
Since there is no integer $s$ with $1\le s<1$,
\cref{thm:main-general} forces $H_{\avec}(1)=1$. Set $D=1+\sum_{i=1}^{d-1}[-a_i]_N$. The first coordinate numerator
of $\lambda(1)$ is $[-D]_N$, and hence
\[
 H_{\avec}(1)=\frac{D+[-D]_N}{N}=\left\lceil\frac DN\right\rceil.
\]
Since $D\ge1$, this equals one precisely when $D\le N$.
For positive $a_i<N$, use $[-a_i]_N=N-a_i$.
\end{proof}

The argument at $t=1$ is the one-row counterpart of the
necessary condition in \cite[Corollary~2.4]{BDS}.

\begin{example}
For $\avec=(4,3,5)$, the heights at $t=1,2,3,4$ are
$1,2,2,3$, so $J_{\avec}=\{1\}$. At $t=3$, the only
possible subtraction gives $H_{\avec}(2)=2$, whereas
$H_{\avec}(3)-1=1$. Thus $(3,2,3)\in2S_{(4,3,5)}$ is not
a sum of two lattice points of $S_{(4,3,5)}$.
\end{example}

The test also yields decompositions. In a positive instance,
record an admissible $s(t)$ for each $t$ of height at least two,
and set $s(t)=t$ at height one. Repeatedly applying
\eqref{eq:no-carry-decomposition} recovers a decomposition of
$p_t$ into $H_{\avec}(t)$ lattice points of $S_{\avec}$.
For an arbitrary $x\in hS_{\avec}\cap\Z^d$, compute
$\lambda=V^{-1}(h,x)$ using \eqref{eq:inverse-vertex-map},
set $n_i=\lfloor\lambda_i\rfloor$, and let $t=[x_d]_N$.
Then
\[
 (h,x)=\sum_{i=0}^d n_i(1,v_i)+V\lambda(t).
\]
Take $n_i$ copies of each $v_i$ and append the decomposition
of $p_t$; when $t=0$, the latter is empty. Their number is
$\sum_{i=0}^d n_i+H_{\avec}(t)=h$. The stated storage bound
concerns the input and preprocessing tables; an explicit list
of $h$ output points has its own output cost.

\section{Constant-coordinate one-row simplices}\label{sec:constant}

Throughout this section, set $m=d-1$. We give a direct proof of
\cref{thm:main-constant} in the original coordinate lattice.
The reduction to a two-dimensional Hilbert basis and the passage
from a planar triangulation to a regular unimodular triangulation
are instances of the two-parameter construction of
\cite[Section~6(a)]{DHH}; see also
\cite[Sections~1.3 and~2.1--2.3]{DLR}. We keep the explicit
calculations to track the lattice index when $\gcd(N,B)>1$
and to include the boundary case $B=0$.

For $\avec=(N-q,\ldots,N-q,N)$, the necessary condition \eqref{eq:first-height-condition} simplifies to $mq+1 \le N$.
Thus, the condition $B = N-mq-1 < 0$ precludes the integer decomposition property (IDP). One can also see the obstruction directly. The point $x=(1,\ldots,1,1)$ satisfies
\[
    x = \frac{q}{N}(e_1+\cdots+e_m) + \frac{1}{N}(N-q,\ldots,N-q,N),
\]
which implies $x \in hS_{d;q,N}$ for $h = \left\lceil (mq+1)/N \right\rceil$. If $B < 0$, then $h \ge 2$. However, the simplex contains no lattice point whose final coordinate is one, as the unique box representative with this final coordinate has height $h > 1$.

Assume henceforth that $B \ge 0$. By reducing the generator in \eqref{eq:general-coordinate-lattice} modulo the integer lattice, we obtain
\begin{equation}\label{eq:constant-lattice}
    \Lambda = \Z^{m+2} + \Z\frac{1}{N}(B,q,\ldots,q,1), \qquad B+mq+1=N.
\end{equation}
The corresponding simplex is $\Delta = \conv(E_0,E_1,\ldots,E_m,E_d)$, with $d=m+1$. Furthermore, the height function is given by the sum of its $m+2$ coordinates.

Define a lattice projection by $\pi(\lambda) = (N\lambda_0,N\lambda_d)$. Its image is the lattice
\[ K_B = \{(s,t) \in \Z^2 : s \equiv Bt \pmod N\}.
\]
Indeed, \eqref{eq:constant-lattice} implies this congruence, and any $(s,t) \in K_B$ is the projection of the lattice vector $(s,qt,\ldots,qt,t)/N \in \Lambda$. The lattice $K_B$ is generated by the basis $\{(N,0), (B,1)\}$ and thus has covolume $N$. The non-negative lattice points in the kernel of $\pi$ are exactly those of the form
\begin{equation}\label{eq:projection-kernel}
    n_1E_1 + \cdots + n_mE_m, \qquad n_i \in \Z_{\ge 0}.
\end{equation}
Indeed, $\lambda_d=0$ forces the coefficient of the fractional generator in \eqref{eq:constant-lattice} to be an integer multiple of $N$. Consequently, any lattice point in the kernel must have integral coordinates, and non-negativity gives the stated form.

\subsection{The ceiling Euclidean algorithm}\label{sec:negative-euclid}

The following is the standard continued-fraction construction of
the Hilbert basis of a two-dimensional cone; compare
\cite[Section~3, especially Theorem~3.16]{DHH}. We spell out
its normalization in $K_B$.

Initialize the sequences by setting $R_0=N$, $R_1=B$, $T_0=0$, $T_1=1$.
For $i \ge 1$, provided $R_i > 0$, recursively define
\begin{equation}\label{eq:negative-euclid}
    a_i = \left\lceil\frac{R_{i-1}}{R_i}\right\rceil, \qquad
    R_{i+1} = a_i R_i - R_{i-1}, \qquad
    T_{i+1} = a_i T_i - T_{i-1}.
\end{equation}
Let $k$ be the minimal index such that $R_k=0$, and define the sequence
\[
    U_i=(R_i,T_i), \qquad 0 \le i \le k.
\]
If $B=0$, the algorithm terminates immediately, yielding $k=1$ and the sequence $U_0=(N,0), U_1=(0,1)$. If $B>0$, the computed quotients $a_1,\ldots,a_{k-1}$ yield the negative continued fraction expansion
\[
    \frac{N}{B} = [a_1,\ldots,a_{k-1}]^-.
\]
Indeed, the relation $R_{i-1}/R_i = a_i - R_{i+1}/R_i$ holds, and iteratively applying this equality produces the expansion. Furthermore, each entry $a_i \ge 2$, since the inequality $R_{i-1} > R_i > 0$ is maintained whenever a quotient is evaluated.

\begin{lemma}\label{lem:planar-hilbert-basis}
Adopting the preceding notation, the following properties hold:
\begin{enumerate}
    \item[(i)] $R_0 > R_1 > \cdots > R_k = 0$ and $0 = T_0 < T_1 < \cdots < T_k$;
    \item[(ii)] If $g = \gcd(N,B)$, then $R_{k-1} = g$ and $T_k = N/g$;
    \item[(iii)] $U_i \in K_B$ and $\det(U_i, U_{i+1}) = N$ for all $0 \le i < k$;
    \item[(iv)] The vectors $U_0, \ldots, U_k$ constitute precisely the Hilbert basis of the semigroup $K_B \cap \R_{\ge 0}^2$.
\end{enumerate}
Furthermore, for each $0 \le i < k$, the linear function
\begin{equation}\label{eq:support-functions}
    \psi_i(s,t) = \frac{(T_{i+1}-T_i)s + (R_i-R_{i+1})t}{N}
\end{equation}
satisfies $\psi_i(U_i) = \psi_i(U_{i+1}) = 1$ and $\psi_i(u) \ge 1$ for every non-zero $u \in K_B \cap \R_{\ge 0}^2$. Consequently, the relative interior of the line segment $[U_i, U_{i+1}]$ contains no lattice points of $K_B$.
\end{lemma}

\begin{proof}
The ceiling function in \eqref{eq:negative-euclid} ensures that $0 \le R_{i+1} < R_i$; hence, the sequence strictly decreases and terminates. The invariant $\gcd(R_{i-1}, R_i) = \gcd(R_i, R_{i+1})$ implies that the final non-zero remainder $R_{k-1}$ is exactly $g$. Because $a_i \ge 2$, we have
\[
    T_{i+1} - T_i = (a_i - 2)T_i + (T_i - T_{i-1}) > 0,
\]
which follows by induction with the base case $T_1 - T_0 = 1$. The same assertions hold when $B=0$.

The initial vectors $U_0$ and $U_1$ belong to $K_B$, and the recurrence shows that all subsequent $U_i$ remain in $K_B$. Their initial determinant is $N$, and
\[
    \det(U_i, U_{i+1}) = \det(U_i, a_i U_i - U_{i-1}) = \det(U_{i-1}, U_i) = N.
\]
At the final pair, this gives $gT_k = N$, establishing the formula for $T_k$. Since this determinant equals the covolume of $K_B$, every consecutive pair $\{U_i, U_{i+1}\}$ forms a lattice basis for $K_B$.

The rays generated by $U_i$ advance counterclockwise from the positive horizontal axis to the positive vertical axis; indeed, their consecutive determinants are positive, their first coordinates strictly decrease, and their second coordinates strictly increase. The cones $\cone(U_i, U_{i+1})$ therefore cover the non-negative quadrant. Any lattice point within one such cone admits a non-negative real combination of the bounding pair; since the pair forms a lattice basis, these coefficients are necessarily integers. Thus, the set $\{U_i\}$ generates the entire semigroup $K_B \cap \R_{\ge 0}^2$.

We now establish that these generators are indecomposable. Equation \eqref{eq:support-functions} and the determinant invariant immediately imply $\psi_i(U_i) = \psi_i(U_{i+1}) = 1$. For a fixed index $i$, the evaluations $\psi_i(U_j)$ satisfy the same linear recurrence as the vectors $U_j$. Starting with two consecutive $1$s and iterating forward, if two successive values satisfy $b \ge a \ge 1$, the subsequent value is $a_j b - a \ge 2b - a \ge b$. The backward recurrence gives the same monotonicity. Therefore, $\psi_i(U_j) \ge 1$ for all $j$. By linearity, since the $\{U_j\}$ generate the semigroup, $\psi_i(u) \ge 1$ for every non-zero $u \in K_B \cap \R_{\ge 0}^2$.

If $U_i$ were decomposable into a sum of two non-zero semigroup elements, the sum of their $\psi_i$-evaluations would be at least $2$, contradicting $\psi_i(U_i) = 1$ (for the case of $U_k$, use $\psi_{k-1}$). Thus, every $U_i$ is indecomposable. Conversely, any non-zero semigroup element that is not a single generator $U_i$ is a sum of at least two generators. This establishes the Hilbert basis assertion.

Finally, any lattice point of $K_B$ lying on the segment $[U_i, U_{i+1}]$ can be uniquely expressed as $\alpha U_i + \beta U_{i+1}$ with $\alpha, \beta \ge 0$ and $\alpha + \beta = 1$. Since $\{U_i, U_{i+1}\}$ is a lattice basis, both coefficients must be integers, restricting the point to the endpoints.
\end{proof}

\subsection{Congruences}

The recurrence below gives the congruence in
\cite[Theorem~2.5]{DLR} and \cite[Theorem~3.6]{Sato}
in the present coordinates. The latter reference assumes
$\gcd(N,B)=1$; the following calculation does not require
that assumption.

\begin{proposition}\label{prop:constant-necessary}
If the simplex $S_{d;q,N}$ possesses the integer decomposition property (IDP), then $B \ge 0$, and every quotient appearing in the sequence \eqref{eq:negative-euclid} must satisfy $a_i \equiv 2 \pmod{m}$.
\end{proposition}

\begin{proof}
We have already shown that $B \ge 0$ is necessary. For each generator $U_i$ in \cref{lem:planar-hilbert-basis}, define the minimal residue
\[
    \rho_i = [q T_i]_N, \qquad \text{and set} \qquad
    Z_i = \frac{1}{N}(R_i, \rho_i, \ldots, \rho_i, T_i).
\]
The defining congruences $R_i \equiv B T_i \pmod{N}$ and $\rho_i \equiv q T_i \pmod{N}$ give $Z_i \in \Lambda$. Furthermore, $Z_i$ is a non-negative vector that projects to $U_i$.

Suppose that $Z_i = x + y$ for some non-negative lattice vectors $x, y \in \Lambda$. Since their projections must sum to the indecomposable basis element $U_i$, at least one of the projected vectors must be zero. By the kernel characterization in \eqref{eq:projection-kernel}, the corresponding summand is a non-negative integral combination of the basis vectors $E_1, \ldots, E_m$. However, because $0 \le \rho_i < N$, every internal coordinate of $Z_i$ is less than one. That summand must therefore be zero. Consequently, $Z_i$ is indecomposable within the unprojected cone, including at the endpoints of the chain.

By \eqref{eq:idp-hilbert-criterion}, IDP forces every indecomposable lattice point to have height one. Thus
\begin{equation}\label{eq:chain-height-one}
    R_i + m\rho_i + T_i = N \qquad (0 \le i \le k).
\end{equation}
For an internal index $1 \le i < k$, scaling the $i$-th equation by $a_i$ and subtracting the adjacent relations for indices $i-1$ and $i+1$ eliminates the terms $R$ and $T$, yielding
\[
    (a_i - 2)N = m(a_i\rho_i - \rho_{i-1} - \rho_{i+1}).
\]
Because $\rho_j \equiv q T_j \pmod{N}$ and the recurrence guarantees $a_i T_i - T_{i-1} - T_{i+1} = 0$, the parenthetical expression on the right-hand side is necessarily divisible by $N$. Canceling $N$ establishes that $a_i - 2$ is divisible by $m$. This argument does not require either $\gcd(N, B) = 1$ or $\gcd(N, m) = 1$.
\end{proof}

\begin{proposition}\label{prop:chain-lifts}
Suppose that $B \ge 0$ and $a_i \equiv 2 \pmod{m}$ for every quotient in \eqref{eq:negative-euclid}. Then the vectors
\[
    Z_i = \frac{1}{N}(R_i, [qT_i]_N, \ldots, [qT_i]_N, T_i) \qquad (0 \le i \le k)
\]
are lattice points of $\Delta \cap \Lambda$. In particular, they satisfy the height condition \eqref{eq:chain-height-one}.
\end{proposition}
\begin{proof}
The sequence of positive sums $R_i+T_i$ has non-negative second differences:
\[(R_{i+1}+T_{i+1}) - 2(R_i+T_i) + (R_{i-1}+T_{i-1}) = (a_i-2)(R_i+T_i) \ge 0.
\]
Their successive differences are therefore non-decreasing. By discrete convexity, any sequence with non-decreasing differences is bounded above by the linear chord connecting its endpoints. Thus the bounds
\[0 < R_i+T_i \le \left(1-\frac{i}{k}\right)(R_0+T_0) + \frac{i}{k}(R_k+T_k) \le N,
\]
since the endpoint values are $R_0+T_0=N$ and $R_k+T_k=N/g \le N$.

The initial sums are $N$ and $B+1=N-mq$, both of which are congruent to $N$ modulo $m$. The recurrence and the hypothesis $a_i \equiv 2 \pmod{m}$ give $R_i+T_i \equiv N \pmod{m}$ for all $i$. Thus, the quantity
\[\widetilde{\rho}_i = \frac{N-R_i-T_i}{m}
\]
is an integer satisfying $0 \le \widetilde{\rho}_i < N$. Its initial values are $\widetilde{\rho}_0=0$ and $\widetilde{\rho}_1=q$. After substituting the recurrence relations into this definition, we get 
\[\widetilde{\rho}_{i+1} = a_i\widetilde{\rho}_i - \widetilde{\rho}_{i-1} - \frac{a_i-2}{m}N.
\]
Reducing modulo $N$ and proceeding by induction from the initial values, we obtain $\widetilde{\rho}_i \equiv qT_i \pmod{N}$. Because $\widetilde{\rho}_i \in [0, N)$, it is the least nonnegative residue: $\widetilde{\rho}_i = [qT_i]_N$. This establishes the identity $R_i + m[qT_i]_N + T_i = N$. As in the proof of \cref{prop:constant-necessary}, this ensures that $Z_i \in \Lambda$. Since the coordinates of $Z_i$ are non-negative and sum to one, we have $Z_i \in \Delta \cap \Lambda$.

When $B=0$, only indices zero and one occur. The assertions follow from the initial values.
\end{proof}

\subsection{Proof of \cref{thm:main-constant}}

We now implement the planar construction of
\cite[Proposition~6.5, Lemma~6.6, Propositions~6.7 and~6.9,
and Corollary~6.10]{DHH} in the coordinates
\eqref{eq:constant-lattice}. The construction itself is standard;
the details below specify the affine lattices, supporting functions,
and determinant normalizations used in this specialization.

Assume for this construction that $m\ge2$ and that the arithmetic
condition in \cref{thm:main-constant}\textup{(iv)} holds. Define
\[
    F = \conv(E_1, \ldots, E_m), \qquad
    c = \frac{1}{m}(E_1 + \cdots + E_m), \qquad
    D = \conv(E_0, c, E_d).
\]
The triangle $D$ is the slice of $\Delta$ on which all middle coordinates are equal. Every lattice point of $\Delta$, other than the vertices $E_1, \ldots, E_m$, lies in $D$. Indeed, the difference between two middle coordinates of a point in $\Lambda$ is an integer. At a point of $\Delta$ other than $E_1,\ldots,E_m$, every middle coordinate is less than one, so all such differences vanish. The remaining vertices $E_0$ and $E_d$ also belong to $D$.

Furthermore, $c \notin \Lambda$: its projection under $\pi$ vanishes, which by \eqref{eq:projection-kernel} would force its middle coordinates to be integers, contrary to their value $1/m$. Define
\[
    \mathcal{A}_D = D \cap \Lambda, \qquad Q = \conv(\mathcal{A}_D).
\]
Since $Z_1$ has positive middle coordinates $q/N$ and does not lie on the segment $[E_0,E_d]$, the convex hull $Q$ is a non-degenerate polygon.

Restricted to the affine plane spanned by $D$, the projection $\pi$ acts as an affine bijection onto the domain
\[
    D_N = \{(s,t) \in \R_{\ge 0}^2 : s+t \le N\},
\]
with inverse
\[
    (s,t) \longmapsto \left(\frac{s}{N}, \frac{N-s-t}{mN}, \ldots, \frac{N-s-t}{mN}, \frac{t}{N}\right).
\]
Under this bijection, $c \mapsto (0,0)$, $E_0 \mapsto (N,0)$, $E_d \mapsto (0,N)$, and $Z_i \mapsto U_i$. Consequently, the image $\pi(\mathcal{A}_D)$ contains all generators $U_i$ and is a subset of $(K_B \cap D_N) \setminus \{(0,0)\}$, although this inclusion may be strict.

\begin{lemma}\label{lem:planar-section-decomposition}
Under the hypotheses of \cref{prop:chain-lifts}, we have the geometric decomposition
\begin{equation}\label{eq:section-decomposition}
D = Q \cup \bigcup_{i=0}^{k-1} \conv(c, Z_i, Z_{i+1}),
\end{equation}
where the interiors of the constituent regions are mutually disjoint. The polygonal chain $[Z_0,Z_1]$, $\ldots$, $[Z_{k-1},Z_k]$ forms the portion of the boundary of $Q$ visible from $c$, encompassing all intermediate collinear lattice points. Furthermore, there exists a regular triangulation of the point configuration $\mathcal{A}_D \cup \{c\}$ in which the simplices incident to $c$ are precisely the triangles $\conv(c, Z_i, Z_{i+1})$, and the remaining triangles constitute a regular unimodular triangulation of $Q$ with respect to its affine lattice.
\end{lemma}
\begin{proof}
For each $i$, \cref{lem:planar-hilbert-basis} guarantees that $\psi_i \ge 1$ on $\pi(\mathcal{A}_D)$ and thus on $\pi(Q)$, with equality on $[U_i,U_{i+1}]$. Several consecutive segments may belong to the same supporting line. Every ray from the origin in the non-negative quadrant falls between the directions of some consecutive pair $U_i, U_{i+1}$ and intersects the segment $[U_i, U_{i+1}]$ exactly once. This intersection point lies in $D_N$ because its endpoints belong to $D_N$. The ray subsequently exits $D_N$ through the boundary segment $[(N,0), (0,N)]$, which is contained entirely within $\pi(Q)$. By convexity, the ray segment connecting the polygonal chain to this exit boundary lies in $\pi(Q)$, whereas the initial ray segment from the origin to the chain is contained in $\conv(0, U_i, U_{i+1})$. This proves the covering in \eqref{eq:section-decomposition}. The inequalities $\psi_i \ge 1$ preclude any interior overlap with $Q$, and the radial order of the $U_i$ gives disjoint triangle interiors.

To construct the regular triangulation, assign a height of zero to all points in $\mathcal{A}_D$ and a height of one to $c$. The zero affine function supports the lifted polygon $Q$ from below. For each $i$, the affine functional
\[
    p \longmapsto 1 - \psi_i(\pi(p))
\]
evaluates to one at $c$, to zero at both $Z_i$ and $Z_{i+1}$, and is non-positive across $\mathcal{A}_D$. It is therefore a lower supporting function. Whenever several consecutive segments of the chain are collinear, their corresponding functionals coincide, merging the associated triangular regions into a single coarse cell. These cells and $Q$ cover $D$, so they form a lower subdivision.

Fix affine Euclidean coordinates on the plane spanning $D$. For a sufficiently small $\eps > 0$, perturb the zero height of each $p \in \mathcal{A}_D$ to $\eps\|p\|^2$, while maintaining the height of $c$ at one. For any fixed $p \in \mathcal{A}_D$, the supporting functional $\eps(2\langle p, x \rangle - \|p\|^2)$ bounds $p$ from below relative to all other points in $\mathcal{A}_D$. Provided $\eps$ is chosen small enough, its value at $c$ remains less than one. Consequently, every point of $\mathcal{A}_D$ appears as a lower vertex in the lifted configuration. The point $c$, being an extreme point of $D$, remains a lower vertex. Applying a further, sufficiently small generic perturbation to the heights of $\mathcal{A}_D$ preserves these properties and, by \cref{lem:small-perturbation}, induces a triangulation that refines the coarse subdivision.

No point of $\mathcal{A}_D$ lies interior to a coarse cell outside of $Q$, as its projection would violate the bound $\psi_i \ge 1$. Furthermore, along the base chain, \cref{lem:planar-hilbert-basis} shows that there are no lattice points between $Z_i$ and $Z_{i+1}$. Therefore, any triangulation of such a coarse cell using all configuration points consists of the triangles $\conv(c, Z_i, Z_{i+1})$. The remaining triangles in the refinement triangulate $Q$ using all of its internal and boundary lattice points. To identify the affine lattice, note that the isomorphism $V$ maps the affine plane of $D$ to the set
\[
\{(1, x, \ldots, x, z) : x, z \in \R\},
\]
whose induced lattice is isomorphic to $\Z^2$ via the coordinates $x, z \in \Z$. Consequently, the affine lattice of $Q$ is isomorphic to $\Z^2$. The emptiness criterion in \cref{lem:polygon} shows that every triangle of $Q$ is unimodular in this lattice. The lifting makes this triangulation regular.
\end{proof}

Let $\mathcal{T}_Q$ denote the triangulation of $Q$ constructed in \cref{lem:planar-section-decomposition}. For $0 \le i < k$, define the simplex
\begin{equation}\label{eq:chain-simplices}
    \Sigma_i = \conv(E_1, \ldots, E_m, Z_i, Z_{i+1}).
\end{equation}
Furthermore, for each triangle $\tau \in \mathcal{T}_Q$ and index $1 \le j \le m$, define
\begin{equation}\label{eq:polygon-simplices}
    \Sigma_{\tau,j} = \conv\bigl(\vertices(\tau) \cup \{E_\ell : 1 \le \ell \le m, \ \ell \ne j\}\bigr).
\end{equation}
By construction, each of the specified vertex sets contains exactly $m+2 = d+1$ points.

\begin{lemma}\label{lem:lifted-volumes}
Every simplex defined in \eqref{eq:chain-simplices} and \eqref{eq:polygon-simplices} is full-dimensional and unimodular with respect to the affine lattice of $\Delta$.
\end{lemma}
\begin{proof}
Since $|\det V| = N$, the normalized volume in the original lattice is exactly $N$ times the absolute determinant of the homogenized vertices in the coordinate lattice. For $\Sigma_i$, expanding this determinant along the standard basis columns $E_1, \ldots, E_m$ leaves a non-zero minor involving only coordinates $0$ and $d$, which evaluates to
\[
\det\begin{pmatrix}
 R_i/N & R_{i+1}/N \\ 
 T_i/N & T_{i+1}/N
 \end{pmatrix}.
\]
Hence
\[
    \nvol(\Sigma_i) = N\frac{|\det(U_i, U_{i+1})|}{N^2} = 1.
\]
Here, $\nvol$ denotes the normalized volume inherited from the original simplex, which is equivalently measured relative to the difference lattice $\Lambda \cap \{\htop=0\}$ parallel to $\aff(\Delta)$.

For $\Sigma_{\tau,j}$, we compute the volume in the original spatial coordinates. The vertices of $\tau$ have the form $(x, \ldots, x, z)$, while the remaining $m-1$ vertices are precisely $e_\ell$ for $\ell \ne j$. Applying the integral linear coordinate transformation
\[(x_1, \ldots, x_m, z) \longmapsto \bigl((x_\ell-x_j)_{\ell \ne j}, x_j, z\bigr)
\]
preserves the lattice structure; it is unimodular, since its inverse recovers $x_j$ and subsequently adds $x_j$ to the other $m-1$ coordinates. After translating a selected vertex of $\tau$ to the origin, we construct the corresponding edge matrix. In the first $m-1$ rows, the columns corresponding to the vectors $e_\ell$ form an identity block, whereas the two edge columns originating from $\tau$ vanish in these rows. Expanding the determinant along this identity block leaves the $2 \times 2$ determinant formed by the two edges of $\tau$ within the $(x,z)$-lattice. By \cref{lem:planar-section-decomposition}, this minor is $\pm 1$, which establishes both full-dimensionality and unimodularity.
\end{proof}

\begin{proposition}\label{prop:constant-triangulation}
Under the hypotheses of \cref{prop:chain-lifts}, with $m\ge2$, the simplices defined in \eqref{eq:chain-simplices} and \eqref{eq:polygon-simplices} constitute a regular unimodular triangulation of $\Delta$, and consequently of $S_{d;q,N}$.
\end{proposition}
\begin{proof}
Let $w$ denote the planar lifting introduced in \cref{lem:planar-section-decomposition}, satisfying $w(c)=1$. We retain these heights for all points in $\mathcal{A}_D$ and assign height one to each basis vertex $E_1, \ldots, E_m$. We show that the specified simplices are lower facets of this lifting.

For a planar triangle $\conv(c, Z_i, Z_{i+1})$, let $\ell$ denote its lower supporting affine functional restricted to the plane of $D$. By definition, $\ell(c)=1$, and equality with the lifting heights is achieved exactly at the three vertices of this triangle. There exists a unique affine functional on the hyperplane spanning $\Delta$ that restricts to $\ell$ on $D$ and evaluates to one on each $E_j$. The prescribed values are consistent by the barycentric relation $c = (E_1 + \cdots + E_m)/m$, which gives $\ell(c)=1$. Explicitly, because an affine functional on this hyperplane is uniquely determined by its evaluations at the basis $\{E_0, E_1, \ldots, E_m, E_d\}$, adopting the evaluations of $\ell$ at $E_0$ and $E_d$ while setting all middle evaluations to one recovers $\ell$ on the affine basis $\{E_0, c, E_d\}$ of $D$. The supporting inequalities at all remaining points of $\mathcal{A}_D$ are inherited directly from the planar configuration. Thus, the contact locus (the subset achieving equality) within the full lattice configuration is precisely $\{E_1, \ldots, E_m, Z_i, Z_{i+1}\}$, verifying that $\Sigma_i$ is a lower facet.

Next, let $\tau \in \mathcal{T}_Q$ with planar lower supporting functional $\ell$. Because $c$ is not a vertex of $\tau$ and the triangulation is generic, we have $\ell(c) < 1$. Fix an index $j$. We extend $\ell$ to the ambient affine hyperplane by assigning a value of one at $E_\ell$ for all $\ell \ne j$. The barycentric relation for $c$ forces its value at $E_j$ to be
\[m\ell(c) - (m-1) < 1.\]
As before, the inequalities across $\mathcal{A}_D$ are governed by the planar bounds. Consequently, this extension is a lower supporting function whose contact locus is exactly the vertex set of $\Sigma_{\tau,j}$. \cref{lem:lifted-volumes} confirms that each such equality set spans a full-dimensional simplex.

We now check that these lower facets cover $\Delta$. Let $\lambda \in \Delta$ and define
\[a = \min_{1 \le i \le m} \lambda_i, \qquad \beta = \sum_{i=1}^m (\lambda_i - a).
\]
Clearly, $0 \le \beta \le 1$. If $\beta=1$, then $\lambda_0 = \lambda_d = a = 0$, forcing $\lambda \in F$, which is contained in every $\Sigma_i$. If $\beta < 1$, define the normalized point
\[p = \frac{1}{1-\beta} (\lambda_0, a, \ldots, a, \lambda_d) \in D.
\]
Selecting an index $j$ that attains the minimum $a$, we decompose $\lambda$ as
\[\lambda = (1-\beta)p + \sum_{\ell \ne j} (\lambda_\ell - a)E_\ell.
\]
The coefficients are non-negative and sum to one. If $p$ falls within a triangle $\tau \in \mathcal{T}_Q$, this convex combination places $\lambda$ in $\Sigma_{\tau,j}$. Alternatively, if $p \in \conv(c, Z_i, Z_{i+1})$, since $c \in F$, the same convex combination places $\lambda$ inside $\Sigma_i$. Because the planar triangles completely cover $D$, these simplices cover $\Delta$.

By the regular subdivision construction in \cref{sec:lifting}, the projections of these lower faces intersect in common faces. Since these full-dimensional lower simplices cover the polytope, they form the regular triangulation induced by the prescribed lifting. Unimodularity follows immediately from \cref{lem:lifted-volumes}.
\end{proof}

\begin{proof}[Proof of \cref{thm:main-constant}]
Suppose that $d \ge 3$. If the simplex is IDP, the necessity of the arithmetic condition follows immediately from \cref{prop:constant-necessary}. Conversely, this arithmetic condition yields the required chain points via \cref{prop:chain-lifts} and induces the regular unimodular triangulation constructed in \cref{prop:constant-triangulation}. Because the existence of a unimodular triangulation guarantees the IDP by \cref{lem:unimodular-idp}, this proves the equivalences, including the boundary case $B=0$.

In dimension $d=2$, every $S_{2;q,N}$ reduces to a two-dimensional lattice polygon, which admits a regular unimodular triangulation by \cref{lem:polygon}. Furthermore, the parameter bound $B = N-q-1 \ge 0$ holds, and the associated congruence condition modulo $m=1$ is trivially satisfied. Therefore, the arithmetic condition is always met, and the theorem holds in this dimension.
\end{proof}

\subsection{An equivalent remainder algorithm}

Remainder criteria for the corresponding quotient lattices also
appear in \cite[Theorem~3.5]{Sato}. We derive the following
two-integer recursion directly from the negative continued fraction.

The following test avoids computing the continued fraction. Fix $m \ge 1$. Let $\mathcal{C}_m(N,q)$ denote the arithmetic condition defined in \cref{thm:main-constant}, extending this notation by setting $\mathcal{C}_m(N,0) = \mathrm{true}$ for all $N \ge 1$.

\begin{proposition}\label{prop:remainder-test}
For any $q > 0$, define $L = mq+1$. Then the condition evaluates recursively as follows:
\begin{equation}\label{eq:remainder-test}
    \mathcal{C}_m(N,q) =
    \begin{cases}
    \mathrm{false}, & N < L, \\
    \mathrm{true}, & N \ge L \text{ and } N \bmod L = 0, \\
    \mathcal{C}_m(r, q \bmod r), & N \ge L \text{ and } r = N \bmod L > 0.
    \end{cases}
\end{equation}
This recursive procedure is guaranteed to terminate. Consequently, this single test simultaneously determines the integer decomposition property (IDP), unimodular triangulability, and regular unimodular triangulability of the simplex $S_{d;q,N}$.
\end{proposition}

\begin{proof}
The base branch $N < L$ corresponds precisely to the established obstruction $B < 0$. If $N = L$, then $B = 0$ and the condition holds trivially. For $N \ge 2L$, the initial quotient in the expansion of $N/(N-L)$ evaluates to exactly two. If $N > 2L$, the subsequent remainder pair is $(N-L, N-2L)$, matching the parameter sequence for the same $q$ but with the parameter reduced to $N-L$. If $N = 2L$, the next remainder is zero, so the condition holds. By induction, one may repeatedly subtract $L$ from $N$. This proves the second branch when $L$ divides $N$, and otherwise reduces the analysis to the case $N = L+r$ with $0 < r < L$.

Apply the Division Algorithm to write $q = hr+u$ with $0 \le u < r$. The initial quotient for the reduced pair $(L+r, r)$ is given by
\[ a_1 = 1 + \left\lceil\frac{L}{r}\right\rceil = 1 + mh + \left\lceil\frac{mu+1}{r}\right\rceil.
\]
The last ceiling lies between $1$ and $m$, since $1 \le mu+1 \le m(r-1)+1 \le mr$. Thus, the congruence $a_1 \equiv 2 \pmod{m}$ holds if and only if this ceiling is one, which is equivalent to the inequality $r \ge mu+1$. If this inequality fails, the recursive call $\mathcal{C}_m(r,u)$ similarly fails its initial branch condition. Conversely, if it holds, we obtain $a_1 = 2+mh$, and the subsequent remainder becomes
\[a_1 r - (L+r) = r - mu - 1.
\]
The remainder of the continued fraction expansion therefore coincides with that of the updated parameters $(r,u)$, interpreting a vanishing remainder as the base case $B=0$. If $u=0$, the parameter pair reduces to $(r, r-1)$, whose quotients are all two, provided the process has not already terminated at $r=1$. This proves the convention $\mathcal{C}_m(r,0) = \mathrm{true}$ and establishes \eqref{eq:remainder-test}.

In any non-terminal recursive step, the parameter strictly decreases: $u = q \bmod r < q$. Indeed, if $r > q$, we would have $u = q$, but the requisite inequality $r \ge mq+1 = L$ would immediately fail because $r < L$. Otherwise, we must have $r \le q$, which forces $u < r \le q$. Consequently, the procedure either terminates or fails immediately, or strictly decreases the positive parameter $q$, so the recursion terminates.
\end{proof}

\begin{corollary}
Fix integers $d\ge2$ and $q\ge1$, and set $m=d-1$, $L=mq+1$. There is an explicitly
computable set $\mathcal R_m(q)\subseteq\{0,1,\ldots,q\}$ such that
all admissible $N$ are exactly
\[
 N=kL+r,\qquad k\ge1,\quad r\in\mathcal R_m(q).
\]
More precisely, we have 
\[\mathcal R_m(q)=\{0\}\cup \{r:1\le r\le q,\ \mathcal C_m(r,q\bmod r)=\mathrm{true}\}.
\]
\end{corollary}
\begin{proof}
We will use \eqref{eq:remainder-test}. If the remainder satisfies $r>q$, then $q\bmod r=q$ and $r<L=mq+1$. Thus the recursive test fails.
All other remainders are classified by the displayed formula.
\end{proof}

For illustration, the resulting residue sets for $d\ge3$ are
\[
\begin{array}{c|c|c}
 q&L&\mathcal R_{d-1}(q)\\ \hline
 1&d&\{0,1\}\\
 2&2d-1&\{0,1,2\}\\
 3&3d-2&\{0,1,3\}\\
 4,\ d=3&9&\{0,1,2,3,4\}\\
 4,\ d\ge4&4d-3&\{0,1,2,4\}.
\end{array}
\]
For example, the remainder $r=2$ for $q=3$ leaves $(N',q')=(2,1)$,
which fails $2\ge m+1$ when $m\ge2$. For $q=4$, the only
nontrivial remaining choice is $r=3$, which leaves $(3,1)$ and
succeeds precisely when $m=2$. The row $q=1$ recovers the
classification $N=kd$ or $N=kd+1$ in \cite[Theorem~A]{BCJ}.

\begin{example}
For the parameters $(d,q,N)=(3,2,7)$, we compute $B=2$ and the negative continued fraction expansion $7/2=[4,2]^-$. Because both entries are congruent to $2$ modulo $m=2$, the associated simplex $\conv(0,e_1,e_2,(5,5,7))$ admits a regular unimodular triangulation.

Conversely, for $(d,q,N)=(3,2,8)$, the expansion evaluates to $8/3=[3,3]^-$. The congruence condition fails, precluding the integer decomposition property (IDP). Explicitly, the lattice point $(3,3,3)$ belongs to the second dilation $2\conv(0,e_1,e_2,(6,6,8))$, yet the final coordinates of the lattice points contained within the original simplex are restricted to the set $\{0,1,4,8\}$. No two values in this set sum to three.

For an example with noncoprime parameters, take $(d,q,N)=(3,2,15)$. Here, $B=10$ and $\gcd(N,B)=5$, yielding the expansion $15/10=[2,2]^-$. The theorem applies directly to establish the IDP, without requiring any preliminary reduction of the original lattice.
\end{example}

\section{Triangulation criteria for one-row simplices}\label{sec:general-triangulations}

The equivalences in \cref{thm:main-constant} use the
constant-coordinate hypothesis. We now return to arbitrary
$\avec=(a_1,\ldots,a_{d-1},N)$. We first specialize two standard
finite tests to the explicitly known lattice configuration, then
recall the Firla--Ziegler obstruction to the converse of
\cref{lem:unimodular-idp}, and finally give a dimension-reduction
construction.

\subsection{A finite compatibility criterion}

The criterion below specializes the defining compatibility condition
for a triangulation and the additivity of normalized volume to
$S_{\avec}$; compare \cite[Definition~2.2.1]{DRS}.
For earlier finite formulations using Hilbert bases and integer
programming, see \cite[Theorem~4 and Corollary~5]{FZ}.
Their facet-balance formulation is different from the compatibility
graph used here.

Using \eqref{eq:all-lattice-points}, we define the full lattice point configuration
\[\mathcal{A}_{\avec} = S_{\avec} \cap \Z^d = \{v_0, \ldots, v_d\} \cup \{p_t : t \in J_{\avec}\}.
\]
This configuration contains $d+1+|J_{\avec}| \le N+d$ elements. Let $\mathcal{U}_{\avec}$ denote the collection of all $(d+1)$-element subsets $\sigma = \{u_0, \ldots, u_d\} \subset \mathcal{A}_{\avec}$ that satisfy the unimodularity condition
\[\left|\det(u_1-u_0, \ldots, u_d-u_0)\right| = 1.
\]
The elements of $\mathcal{U}_{\avec}$ serve as candidate vertex sets of maximal simplices: while the vertex set of every maximal simplex of a unimodular triangulation necessarily belongs to $\mathcal{U}_{\avec}$, individual membership does not guarantee extendability to a complete triangulation.

Two distinct candidate vertex sets $\sigma, \tau \in \mathcal{U}_{\avec}$ are said to be \emph{compatible} if their convex hulls intersect in a common face, a condition expressed formally as
\begin{equation}\label{eq:compatibility}
    \conv(\sigma) \cap \conv(\tau) = \conv(\sigma \cap \tau),
\end{equation}
with the standard convention that $\conv(\emptyset) = \emptyset$. We define the compatibility graph $G_{\avec}$ to have vertex set $\mathcal{U}_{\avec}$, with edges connecting precisely the compatible pairs. A \emph{clique of size $N$} is a set of $N$ vertices of this graph such that every pair is connected by an edge; here it represents $N$ pairwise compatible candidate vertex sets.

\begin{theorem}\label{thm:compatibility-test}
The simplex $S_{\avec}$ admits a unimodular triangulation if and only if the compatibility graph $G_{\avec}$ contains a clique of size $N$. Furthermore, the unimodular triangulations of $S_{\avec}$ correspond precisely to the cliques of size $N$ in $G_{\avec}$.
\end{theorem}
\begin{proof}
Every unimodular triangulation has exactly $N$ maximal simplices,
since $\nvol(S_{\avec})=N$, and their vertex sets form a clique by
\eqref{eq:compatibility}. Conversely, the convex hulls of the vertex sets in a clique
of size $N$ have disjoint interiors and total normalized volume
$N$. Their union is closed in $S_{\avec}$ and has its full
volume, so it equals $S_{\avec}$: a nonempty relatively open
complement would have positive $d$-dimensional volume.
Compatibility then gives a unimodular triangulation by adjoining
all faces.
\end{proof}

Compatibility of a pair $\sigma, \tau \in \mathcal{U}_{\avec}$ can be tested by rational linear programming. We first consider the linear feasibility problem
\begin{equation}\label{eq:intersection-lp}
    \sum_{p \in \sigma} \alpha_p p = \sum_{q \in \tau} \beta_q q, \qquad 
    \sum_{p \in \sigma} \alpha_p = \sum_{q \in \tau} \beta_q = 1, \qquad
    \alpha_p \ge 0, \qquad \beta_q \ge 0.
\end{equation}
If the system is infeasible, the simplices $\conv(\sigma)$ and $\conv(\tau)$ are disjoint and therefore compatible. If it is feasible, we maximize the objective function $\sum_{p \in \sigma \setminus \tau} \alpha_p$ over the feasible region. The maximum is zero if and only if the pair is compatible. Indeed, the uniqueness of non-negative barycentric coordinates ensures that a point belongs to the common face $\conv(\sigma \cap \tau)$ precisely when all weights outside this face vanish. Thus, a zero maximum guarantees that the geometric intersection is confined to the common face, since the reverse inclusion in \eqref{eq:compatibility} always holds. (In the extreme case where $\sigma \cap \tau = \emptyset$ yet the problem is feasible, the objective is one, so the pair is incompatible.) Because the feasible region is bounded, the maximum is always attained whenever the system is feasible.

Since all input coefficients are integers, feasibility and optimality can be determined using exact rational arithmetic (e.g., by evaluating the basic feasible solutions of the linear system). Algorithmically, this procedure tests at most $\binom{N+d}{d+1}$ candidate subsets before performing a finite graph search for the requisite clique. We claim no polynomial-time bound for the graph search.

The next criterion is the usual linear feasibility test for
regularity; see \cite[Section~2.3, p.~71]{DRS} and
\cite[Section~1.1, Equations~(1.1)--(1.2)]{HPPS}.

\begin{proposition}\label{prop:regularity-test}
Let $\mathcal{T}$ be a triangulation arising from \cref{thm:compatibility-test}. Then $\mathcal{T}$ is regular if and only if the following finite rational linear system is feasible. Specifically, introducing a height variable $w_p$ for each $p \in \mathcal{A}_{\avec}$ and an affine function $\ell_\Sigma: \R^d \to \R$ for each maximal simplex $\Sigma \in \mathcal{T}$, the system requires
\begin{equation}\label{eq:regularity-system}
 \ell_\Sigma(p) = w_p \quad (p \in \vertices(\Sigma)), \qquad
 \ell_\Sigma(q) \le w_q - 1 \quad (q \in \mathcal{A}_{\avec} \setminus \vertices(\Sigma)).
\end{equation}
Consequently, enumerating the cliques of size $N$ and evaluating this feasibility condition provides a finite decision procedure for the existence of a regular unimodular triangulation.
\end{proposition}
This is \eqref{eq:strict-lifting} with the strict inequalities
normalized to gaps of at least one. Indeed, a regular lifting has
a positive minimum gap over the finitely many pairs $(\Sigma,q)$;
dividing the heights and supporting functions by that minimum
gives \eqref{eq:regularity-system}. The reverse implication is
immediate from the definition of a regular triangulation.

\begin{remark}
For \cite[Question~3.13]{BCJ}, \cref{thm:main-general} gives
an arithmetic IDP test in terms of the box heights. By contrast,
\cref{thm:compatibility-test,prop:regularity-test} apply standard
finite searches to the explicit point set
\eqref{eq:all-lattice-points}. They do not constitute a structural
classification of all one-row simplices admitting unimodular
triangulations. The example below also shows why the IDP test
cannot serve as such a classification.
\end{remark}

\subsection{An IDP simplex without a unimodular triangulation}

We state the external obstruction used here in its cone form. A \emph{Hilbert partition} of a pointed lattice cone is a triangulation of the cone into unimodular cones generated by
subsets of its Hilbert basis.

\begin{example}(Firla--Ziegler, \cite[Example~10]{FZ})\label{prop:FZ}
The cone
\[C=\cone\bigl(e_1,e_2,e_3,(14,31,34,39)\bigr)\subset\R^4
\]
has no Hilbert partition. Its Hilbert basis is contained in the height-one hyperplane $y_1+y_2+y_3-2y_4=1$.
\end{example}

Both the Hilbert-basis assertion and the nonexistence of a
Hilbert partition are due to Firla and Ziegler. The following
example identifies their cone with the homogenized cone
of a one-row simplex. The equivalent cyclic quotient example is
also recorded in \cite[Note~6.2(iii)]{DHZ}.

\begin{example}\label{ex:39}
The one-row simplex
\[
 S=\conv\bigl(0,e_1,e_2,(34,31,39)\bigr)
\]
is IDP and has no unimodular triangulation.
\end{example}
\begin{proof}[Identification with the Firla--Ziegler example]
The integral linear map
\[
 \Phi(h,x_1,x_2,x_3)
 =(h-x_1-x_2+2x_3,\ x_2,\ x_1,\ x_3)
\]
is unimodular, with inverse
\[
 (y_1,y_2,y_3,y_4)\longmapsto
 (y_1+y_2+y_3-2y_4,\ y_3,\ y_2,\ y_4).
\]
It sends the homogenized vertices of $S$ to
$e_1,e_3,e_2,(14,31,34,39)$ and identifies height with
$y_1+y_2+y_3-2y_4$. By \cref{prop:FZ} and
\eqref{eq:idp-hilbert-criterion}, $S$ is IDP. Its height-one
lattice points are therefore precisely the Hilbert basis of its
cone. A unimodular triangulation of $S$ would give, after coning
and applying $\Phi$, a Hilbert partition of $C$, contrary to
\cref{prop:FZ}.
\end{proof}

\subsection{A dimension-reduction construction}

The next equivalence uses a subdivision into two
pyramids of lattice height one. Coning over a unimodular base
preserves unimodularity; compare the lattice-pyramid discussion
in \cite[Section~1.1]{HPPS}. We give the coordinate calculation
and a common lifting for the two pyramids. The first coordinate
can be replaced by any of the first $d-1$ coordinates.
For this statement, we also use $S_{\bvec}$ when its last
parameter is $M=1$, in which case it is a unimodular simplex;
in dimension one, $S_{(M)}$ denotes the interval $[0,M]$.

\begin{theorem}\label{thm:dimension-reduction}
Assume that $d \ge 2$, $M \ge 1$, and $0 \le c_i < M$ for all $2 \le i \le d-1$. Define the parameter vectors
\[\avec = (2M-1, 2c_2, \ldots, 2c_{d-1}, 2M), \qquad
\bvec = (c_2, \ldots, c_{d-1}, M).
\]
For each of the following properties, $S_{\avec}$ has the property
if and only if $S_{\bvec}$ has it: IDP, admitting a unimodular
triangulation, and admitting a regular unimodular triangulation.
For $d=2$, the base simplex is the interval $[0,M]$.
\end{theorem}
\begin{proof}
Let $f(x)=x_1-x_d$ and $H=\ker f$. The integral functional $f$ is \emph{primitive}, meaning that $f(\Z^d)=\Z$. Among the vertices of $S_{\avec}$, precisely $e_1$ and $v_d = \avec$ lie outside $H$. They satisfy $f(e_1) = 1$ and $f(v_d) = -1$. The segment connecting them intersects $H$ exactly at its midpoint,
\[w = \frac{e_1+v_d}{2} = (M, c_2, \ldots, c_{d-1}, M) \in \Z^d.
\]
The hyperplane section is therefore
\[ F = S_{\avec} \cap H = \conv(0, e_2, \ldots, e_{d-1}, w).
\]
These are all the vertices of the section because an edge of a simplex crosses $H$ at an interior point if and only if its endpoints yield $f$-values of opposite signs; the unique edge satisfying this condition is $[e_1,v_d]$. Alternatively, expressing the barycentric coordinates of a point in $S_{\avec}$ allows one to isolate the symmetric contribution of $e_1$ and $v_d$, which corresponds to $w$, giving
\begin{equation}\label{eq:two-unit-pyramids}
    S_{\avec} = \conv(F,e_1) \cup \conv(F,v_d), \qquad \conv(F,e_1) \cap \conv(F,v_d) = F.
\end{equation}
The projection map
\[H \cap \Z^d \longrightarrow \Z^{d-1}, \qquad (x_1, x_2, \ldots, x_d) \longmapsto (x_2, \ldots, x_d)
\]
defines a lattice isomorphism, since $x_1 = x_d$ holds on $H$. This isomorphism identifies $F$ with $S_{\bvec}$.

Given a unimodular triangulation of $F$, taking the geometric cone over each of its maximal simplices with respect to the apices $e_1$ and $v_d$ produces a face-to-face triangulation of $S_{\avec}$, as dictated by \eqref{eq:two-unit-pyramids}. To see that each resulting simplex is unimodular, note that the edge vectors of its base form a lattice basis for $H \cap \Z^d$, while the vector connecting any base vertex to either apex evaluates to $\pm 1$ under $f$. Because $f: \Z^d \to \Z$ is surjective, the full set of vectors forms a $\Z$-basis for $\Z^d$.

To prove regularity, we construct a lifting. By convexity, the value of $f$ at any lattice point in $S_{\avec}$ must be an integer in $[-1,1]$. Moreover, the extreme values $\pm 1$ are achieved uniquely at the vertices $e_1$ and $v_d$. Hence, all remaining lattice points of $S_{\avec}$ belong to $F$. Suppose $\eta$ is a lifting that induces the given regular triangulation of $F$. We define a new height function on $S_{\avec} \cap \Z^d$ by assigning heights $\eps \eta(p)$ to points $p \in F$, and a height of $1$ to both apices $e_1$ and $v_d$. Here, $\eps > 0$ is chosen sufficiently small to ensure that $\eps \eta(w) < 1$.

For any maximal simplex $\tau$ in the base triangulation, let $\ell_\tau$ be the affine functional supporting the lifted vertices of $\tau$ under $\eta$, satisfying strict inequalities at all other lattice points of $F$. Extend the functional $\eps \ell_\tau$ affinely to $\R^d$ by stipulating that it takes the value $1$ at $e_1$. Since $w = (e_1+v_d)/2$, its evaluation at $v_d$ satisfies
\[2\eps\ell_\tau(w)-1 \le 2\eps\eta(w)-1 < 1.
\]
This extension is a lower supporting function whose equality set is the vertex set of $\conv(\tau,e_1)$. A symmetric construction applies to the cone over $v_d$ by instead prescribing the value $1$ at $v_d$. By \eqref{eq:two-unit-pyramids}, these lower simplices completely cover $S_{\avec}$, so the triangulation is regular.

Conversely, let $\mathcal T$ be a unimodular triangulation of
$S_{\avec}$. As observed above, its vertices outside $F$ can only
be $e_1$ and $v_d$. No simplex of $\mathcal T$ contains both:
their edge contains the lattice midpoint $w$, contradicting
\cref{lem:unimodular-idp}. Every maximal simplex therefore has
one apex and $d$ vertices in $F$, and lies in one of the two
pyramids. Its intersection with $F$ is a facet. These facets
triangulate $F$; the unit lattice height of the apex shows that
each has the same normalized volume as the full simplex, namely
one. If $\mathcal T$ is regular, restrict its lifting to
$F\cap\Z^d$. A lower supporting function for a maximal simplex
restricts to a supporting function for its base facet, with
strict inequalities at all other points of $F$. Since these
facets cover $F$, the induced triangulation of $F$ is regular.

It remains to prove the IDP equivalence without assuming a
triangulation. Suppose first that $F$ is IDP, and let
$x\in hS_{\avec}\cap\Z^d$, $h\ge1$. Put $k=f(x)\in\Z$.
If $k\ge0$, the positive pyramid in \eqref{eq:two-unit-pyramids}
gives $0\le k\le h$ and
$x-ke_1\in(h-k)F\cap\Z^d$. Decomposing this point in $F$
and adjoining $k$ copies of $e_1$ gives $h$ summands.
If $k<0$, use $x-(-k)v_d\in(h+k)F$ and $-k$ copies of $v_d$.
When the remaining dilation is zero, the remaining point is zero
and contributes no summands.
Conversely, suppose that $S_{\avec}$ is IDP and take
$x\in hF\cap\Z^d$. In a decomposition into $h$ lattice points
of $S_{\avec}$, all summands outside $F$ are the two apices.
Their multiplicities are equal because $f(x)=0$ and their
$f$-values are $1$ and $-1$. Replace each pair by two copies of
$w$, using $e_1+v_d=2w$. This gives a decomposition into $h$
lattice points of $F$. The lattice isomorphism above transfers
all three equivalences from $F$ to $S_{\bvec}$.
\end{proof}

\begin{corollary}
For any integer $N \ge 3$, define the simplex
\[S_N = \conv\bigl(0, e_1, e_2, (N-1, N-2, N)\bigr).
\]
Then $S_N$ possesses the integer decomposition property (IDP) if and only if $N$ is even. Furthermore, $S_N$ admits a regular unimodular triangulation when $N$ is even, and no unimodular triangulation when $N$ is odd.
\end{corollary}

\begin{proof}
If $N=2M$ is even, \cref{thm:dimension-reduction} reduces the problem to the lattice triangle $S_{(M-1,M)}$. By \cref{lem:polygon}, this triangle admits a regular unimodular triangulation. Consequently, $S_N$ also admits such a triangulation, and is therefore IDP by \cref{lem:unimodular-idp}.

Now suppose that $N=2k+1$ is odd. In \eqref{eq:general-coordinate-lattice}, the generator numerator is congruent to $(-4,1,2,1)$ modulo $N$. Thus the height function is
\[H(t) = \frac{[-4t]_N+2t+[2t]_N}{N} = \left\lceil\frac{2t+[2t]_N}{N}\right\rceil.
\]
For $1 \le t \le k$, we have $[2t]_N = 2t$, meaning the condition $H(t) = 1$ is equivalent to $4t \le N$. For $k+1 \le t < N$, the expression $2t+[2t]_N = 4t-N$ strictly exceeds $N$, implying that $H(t) \ge 2$. Therefore, the index set of junior points is precisely
\[J = \left\{1, \ldots, \lfloor N/4 \rfloor \right\}.
\]
At $t_0 = k+1 = (N+1)/2$, we have $[2t_0]_N = 1$, which gives $H(t_0) = 2$. By the last coordinate, any decomposition of this height-two box point into a sum of two junior points would require $t_0 = s_1 + s_2$ for some $s_1, s_2 \in J$. However, this yields
\[s_1 + s_2 \le 2\lfloor N/4 \rfloor < \frac{N+1}{2} = t_0,
\]
which is a contradiction. The geometric obstruction in spatial coordinates corresponds to the lattice point $(k+1,k,k+1) \in 2S_N$. Thus, $S_N$ lacks the IDP for odd $N$, and \cref{lem:unimodular-idp} consequently precludes the existence of any unimodular triangulation.
\end{proof}

\section{Concluding remarks}\label{Section-Finily}

\cref{thm:main-ehrhart} answers \cite[Question~5.14]{BCJ} and determines the log-concave and real-rooted cases. In particular,
the three leading coefficients already detect failure of log-concavity for $d\ge6$ and failure of real-rootedness for
$d\ge4$.

For the constant-coordinate family, \cref{thm:main-constant} expresses the two-parameter theory of \cite{DHH,DLR,Sato}
in Hermite normal form coordinates, giving the classification requested in \cite[Question~3.14]{BCJ}. The proof tracks the
original lattice when $\gcd(N,B)>1$ and gives explicit regular triangulations. The equivalent remainder test in
\cref{prop:remainder-test} describes, for fixed $d,q$, all admissible $N$ by finitely many residue classes above the
necessary threshold.

For arbitrary one-row simplices, \cref{thm:main-general} specializes the Hilbert-basis criterion to a finite arithmetic
test, with decompositions or certificates of failure. The standard compatibility and lifting tests in
\cref{thm:compatibility-test,prop:regularity-test} are finite decision procedures, while \cref{thm:dimension-reduction}
provides an additional constructive sufficient condition.
A structural classification of unimodular triangulations for the general family remains separate from these tests. The
Firla--Ziegler example \cite[Example~10]{FZ}, written in our coordinates in \cref{ex:39}, shows that such a classification
must distinguish triangulability from IDP.






\noindent
{\small \textbf{Acknowledgments:}}
The authors would like to express sincere gratitude for all the suggestions that have improved the presentation of this paper.
Feihu Liu was partially supported by the Postdoctoral Fellowship Program and China Postdoctoral Science Foundation (Grant No. BX2026002).

\noindent{\small \textbf{Declaration of AI Assistance:}}
During the preparation of this manuscript, the authors used ChatGPT to assist in exploring possible approaches, checking technical details, and improving the exposition. All mathematical arguments, computations, proofs, and results were independently verified by the authors. The authors take full responsibility for the content of this manuscript.

\end{document}